\documentclass[preprint,12pt]{elsarticle}

\usepackage{amsfonts,amssymb,amsmath,amsthm,bm}
\usepackage{mathtools}
\usepackage{dsfont}
\usepackage{geometry}                		
\usepackage{subfigure}

\newtheorem{lemma}{Lemma}
\newtheorem{theorem}{Theorem}

\newtheorem{remark}{Remark}
\newtheorem{definition}{Definition}

\makeatletter
\def\ps@pprintTitle{%
  \let\@oddhead\@empty
  \let\@evenhead\@empty
  \let\@oddfoot\@empty
  \let\@evenfoot\@oddfoot
}
\makeatother

\begin{document}

\begin{frontmatter}

\title{The Feller diffusion conditioned on a single ancestral founder}
\author[label1]{Conrad J.\ Burden}
\ead{conrad.burden@anu.edu.au}
\author[label2]{Robert C.\ Griffiths}
\ead{Bob.Griffiths@Monash.edu}
\address[label1]{Mathematical Sciences Institute, Australian National University, Canberra, Australia}
\address[label2]{School of Mathematics, Monash University, Australia}


\begin{abstract}
We examine the distributional properties of a Feller diffusion $\big(X(\tau)\big)_{\tau \in [0, t]}$ conditioned on the current population $X(t)$  
having a single ancestor at time zero.  The approach is novel and is based on an interpretation of Feller's original solution according to which the
current population is comprised of a Poisson number of exponentially distributed families, each descended from a single ancestor.  The distribution of 
the number of ancestors at intermediate times and the joint density of coalescent times is determined under assumptions of initiation of the process 
from a single ancestor 
at a specified time in the past, including infinitely far in the past, and for the case of a uniform prior on the time since initiation.  Also calculated are 
the joint distribution of the time since the most recent common ancestor of the current population and the contemporaneous population size at that 
time under different assumptions on the time since initiation.  In each case exact solutions are given for supercritical, critical and subcritical 
diffusions.  For supercritical diffusions asymptotic forms of distributions are also given in the limit of unbounded exponential growth.  
\end{abstract}

\end{frontmatter}

%

\section{Introduction}
\label{sec:Introduction}

In a recent paper~\citep{BurdenGriffiths24} we characterised the stochastic properties of the coalescent tree of a Feller diffusion 
conditioned on an observed final population size.  Our approach was based on recognising the Feller diffusion as the limit 
of a continuous-time, finite-population birth-death (BD) process in which the birth and death rates become infinite, but their difference 
remains finite.  

The coalescent tree of a BD process stopped at finite time is generally referred to as a reconstructed process~\citep{Nee94}, 
and has been studied extensively. \cite{Aldous05} and \citet{Gernhard08} have analysed the reconstructed trees conditioned on a fixed number of 
leaves for critical and supercritical BD processes respectively.  
\citet{Stadler09} and \citet{Wiuf18} have extended the analysis to the reconstructed trees of sampled populations, while \citet{Stadler12} 
calculated distributions of branch lengths in reconstructed trees conditioned on a fixed number of leaves.  
\citet{Lambert13b} study general classes of BD processes by interpreting the reconstructed process as a coalescent point process 
in which a sequence of identically and independently distributed (i.i.d.) node heights is stopped at the first node height exceeding 
the time since initiation from a single ancestral founder.  This construction lends itself to studying Bernoulli sampled populations, but is not 
immediately applicable to the genealogy of a sample of fixed size~\citep{Lambert18}.  \citet{Lambert13c} consider the effect of tip removals 
to study contemporary species extinctions. 
\citet{o1995genealogy} and \cite{Harris20} have determined coalescent times of a BD process initiated from 
a single founder in the near critical limit, which can also be identified with a Feller diffusion. 
For a review of coalescence in Feller diffusion processes, see \citet{BurdenGriffiths25a}. 

In \citet[Section~5]{BurdenGriffiths24} we showed that the process generating that part of the coalescent tree of a Feller diffusion $\big(X(\tau)\big)_{\tau \in [0, t]}$ 
up to a time $t - 2\epsilon$ is identical to a process generating the reconstructed tree of a constant-rate BD process stopped at time $t - 2\epsilon$.  The derived birth and 
death rates of the equivalent BD process, $\hat\lambda(\epsilon)$ and $\hat\mu(\epsilon)$, become infinite as $\epsilon \to 0$.  By exploiting the known results for the 
distribution of coalescent times in a BD process conditioned on its final population size and taking carefully scaled limits the joint density of the infinite set of 
coalescent times for a Feller diffusion conditioned on $X(t) = x$ can be determined.  

The current paper takes a more direct approach and leads to a number of new results pertaining to the coalescent tree of a Feller diffusion.  
We begin with the solution to the forward Kolmogorov equation originally provided by \citet[Eq.~5.2]{feller1951diffusion}, 
which takes the form of a Poisson mixture of gamma distributions.  Within this solution, the `current' population $X(t)$ is interpreted as a sum of exponentially distributed 
families, each descended from a single ancestor at time zero. The number of founding family ancestors is Poisson distributed. Given this interpretation, and the Markovian 
nature of the diffusion, it is relatively straightforward to calculate distributional properties of the population at any intermediate time conditional on initial and 
final populations, and hence explore the coalescent structure.  

We begin in Section~\ref{sec:Background} with a brief summary of the relevant properties of the Feller diffusion.  A more detailed summary is given in 
\cite{BurdenGriffiths24}.  In Section~\ref{sec:singleAncestor} we concentrate on the properties of a Feller diffusion $\big(X(\tau)\big)_{\tau \in [0, t]}$ conditioned 
on a final population $X(t) = x > 0$ descended from a single founding ancestor at time zero.  A seminal result is Lemma~\ref{lemma:IntermediateAncestors} which 
gives the distribution of the number of ancestors at intermediate times within the interval $[0, t]$.  
In Section~\ref{sec:singleAncestor} the time $t$ since initiation of the process is a free parameter.  Subsequent 
sections consider particular cases.  Section~\ref{sec:TInfLimit} deals with the limiting distributions in which the time since initiation with a single ancestral founder 
becomes infinite.  This limit is pertinent to any gene whose ancestral lineage has its origins in the far distant past.  On the other hand, many studies of reconstructed 
processes in the context of finite BD processes assume an improper uniform prior on the time since initiation. Section~\ref{sec:UnifT1} is devoted to setting an 
improper uniform prior on the time since initiation of a Feller diffusion with a single founding ancestor.  By contrast, Section~\ref{sec:improperUnifX0} deals with 
setting an improper uniform prior on the population size $X(0)$ at a fixed time $t$ in the past without restricting to a single ancestral founder.  We demonstrate a 
mathematical equivalence of this scenario to the limit analysed in Section~\ref{sec:TInfLimit}.  

In earlier work, \citet{burden2016genetic} and \citet{BurdenSoewongsoso19} addressed the problem of characterising the joint likelihood of the time since the 
most recent common ancestor (MRCA) of a currently observed population $X(t) = x$ and the contemporaneous population size at the time of that MRCA.  
In Section~\ref{sec:MRCAgivenX} we revisit this problem using methods developed in the current paper.  The result is a more rigorous treatment leading to 
concise analytic formulae for each of the scenarios postulated in Sections~\ref{sec:singleAncestor}, \ref{sec:TInfLimit} and~\ref{sec:UnifT1}.  

A supercritical Feller diffusion which survives extinction has a tendency towards unbounded exponential growth.  Limiting distributions pertaining to coalescent 
times and the population size at the time of the MRCA under this scenario are determined in Section~\ref{sec:LargeXDistribs}.  Finally, 
conclusions are discussed in Section~\ref{sec:Conclusions}.  

%

\section{Background material and notation}
\label{sec:Background}

%

\subsection{Feller diffusion and solution given initial population}
\label{sec:FellerDiff}

The generator of the Feller diffusion $X(t)$ is 
\begin{equation}	\label{FellerGenerator}
\mathcal{L} = \tfrac{1}{2} x \frac{\partial^2}{\partial x^2}+ \alpha x \frac{\partial}{\partial x}, 
\end{equation}
where $\infty < \alpha < \infty$.  The diffusion is said to be subcritical, critical or supercritical according as $\alpha <0$, $= 0$ or $> 0$ respectively.  
For $x_0 \in \mathbb{R}_{\ge 0}$, define the probability density $f(x_0, x; t)$ by 
\[
f(x_0, x; t) dx = \mathbb{P}(X(t) \in(x, x + dx) \mid X(0) = x_0), \qquad x, t \ge 0.
\]
Solving the forward equation via a Laplace transform in $x$ \citetext{\citealp[Section~5.11]{Cox78}; \citealp[Section~3]{BurdenGriffiths23}} yields the solution as a Poisson-gamma mixture 
\begin{equation}	\label{Poisson_Gamma}
f(x_0, x; t) = \delta(x) e^{-x_0 \mu(t)} + \sum_{l = 1}^\infty f_l(x_0, x; t), \qquad x, t \ge 0, 
\end{equation}
where
\begin{equation}	\label{flDef}
f_l(x_0, x; t) := \frac{(x_0 \mu(t))^l}{l!}e^{-x_0 \mu(t)}  \frac{x^{l - 1} e^{-x/\beta(t)}}{\beta(t)^l (l - 1)!}, 
\end{equation}
\begin{equation}	\label{mubetaDef}
\mu(t; \alpha) = \frac{2\alpha e^{\alpha t}}{e^{\alpha t} - 1}, \quad \beta(t; \alpha) = \frac{e^{\alpha t} - 1}{2\alpha}, \qquad \alpha \ne 0, 
\end{equation} 
and $\delta(x)$ is the Dirac $\delta$-function \citep[Section~15]{Dirac58}.  We set $\mu(t; 0) = 2/t$ and $\beta(t; 0) = t/2$.  
The first term in Eq.~(\ref{Poisson_Gamma}) accounts for extinction of the entire population up to time $t$.  Some useful identities involving the functions defined in Eq.~(\ref{mubetaDef}) are listed in \ref{sec:MuBetaIdentities}.  In general, we will suppress the $\alpha$ dependence in in these two functions and simply write 
$\mu(t)$ and $\beta(t)$ unless necessary for context, for instance, in Eqs.~(\ref{mubetaIdentity3}) and (\ref{mubetaIdentity4}).     
Alternate forms of Eq.~(\ref{Poisson_Gamma}) are given in \ref{sec:Alternate_f}.  

%

\subsection{Diffusion limits of finite population processes}
\label{sec:DiffusionLimits}

The Feller diffusion arises as a near-critical limit of either a continuous-time BD process or a Bienaym\'{e}-Galton-Watson (BGW) process.  
The interpretation of Eq.~(\ref{Poisson_Gamma}) following from either of these limits is that the number of ancestors $l$ at time $0$ of the final population $X(t)$ 
is Poisson distributed with mean $x_0 \mu(t)$, and the size of each family is identically and independently exponentially distributed with mean $\beta(t)$.  
The probability of extinction of the entire population occurring in the interval $[0, t]$ is $e^{-x_0\mu(t)}$.  A detailed description of 
the diffusion limits and the justification for this interpretation can be found in \citet[Section~2]{BurdenGriffiths24}. A brief summary of the diffusion limits follows.  

Consider a linear BD process $(M_\epsilon(t))_{t \in \mathbb{R}_{\ge 0}}$ with birth rate $\hat\lambda(\epsilon)$ and death rate $\hat \mu(\epsilon)$ specified as functions of a parameter $\epsilon \in \mathbb{R}_{> 0}$ with the property 
\begin{equation}	\label{BDlimit}
\begin{split}
\hat\lambda(\epsilon) &= \tfrac{1}{2} \epsilon^{-1} + \tfrac{1}{2} \alpha + \mathcal{O}(\epsilon), \\
\hat \mu(\epsilon)        &= \tfrac{1}{2}\epsilon^{-1} - \tfrac{1}{2} \alpha + \mathcal{O}(\epsilon), 
\end{split}
\end{equation}
as $\epsilon \to 0$.  
If $M_\epsilon(t)$ is the number of particles alive at time $t$, set $X(t) = \epsilon M_\epsilon(t)$.  
Then the limiting generator of the process $X(t)$ as $\epsilon \to 0$ is the Feller diffusion generator Eq.~(\ref{FellerGenerator}).  

Alternatively, consider a BGW process $\big(Y(i)\big)_{i \in \mathbb{Z}_{\ge 0}}$ with parameters $\lambda$ equal to the 
expected number of offspring per parent, $\sigma^2$ equal to the variance of the number of offspring per parent, and a non-zero probability that a 
parent has no offspring.  Define a scaled time $t \ge 0$ and scaled population $X(t)$ by 
\begin{equation}	\label{BGWscaling}
\alpha t := i \log\lambda, \qquad \alpha X(t) := Y(i) \frac{\log\lambda}{\sigma^2}.  
\end{equation}
Then the limiting generator of the process as $\lambda \to 1$, $\sigma^2$  fixed and $i \to \infty$ in such a way that $\alpha t$ remains finite is the Feller 
generator Eq.~(\ref{FellerGenerator}).  
Simultaneously, the number of particles becomes infinite in the limit in such a way that $\alpha X(t)$ remains finite.  Note that the 
parameter $\alpha$ introduces one degree of redundancy as only the combinations $\alpha t$ and $\alpha X(t)$ are set by Eq.~(\ref{BGWscaling}).  

In \citet[Eq.~(6)]{BurdenGriffiths24}, $\alpha$ is determined by setting the initial condition $X(0) = 1$ for the entire initial population of 
$y_0 := Y(0)$ individuals.  In the diffusion limit $y_0 \to \infty$, and hence $\alpha = \lim_{y_0 \to \infty, \lambda \to 1} (y_0 \log\lambda)/\sigma^2$.  
In that paper the initial condition $X(0) = x_0 \in [0, 1]$ is reserved 
for the limiting process descended from a subset of the initial population of size $x_0 y_0$.

In a general setting, however, when specifying an initial condition 
$X(0) = x_0$ to the process defined by the generator Eq.~(\ref{FellerGenerator}), $x_0$ can take any nonnegative value.  
As an alternative to setting $X(0)$, without loss of generality the parameter $\alpha$ can always be set to $+1$ for a supercritical diffusion, 
$-1$ for a subcritical diffusion, or $0$ for a critical diffusion, and 
a boundary condition can be set by specifying a distribution over the range $[0, \infty)$ for $X(t)$ at a chosen value of $t$.  

In the current paper both $\alpha$ and $X(0) = x_0$ are treated as free parameters of the model.  Taking the limit $x_0 \to 0$ will prove to be a 
convenient way to condition on a Feller diffusion generating final population with a single initial ancestor.  
By leaving $\alpha$ 
unspecified, results for a critical diffusion can easily be obtained by taking the limit $\alpha \to 0$.  The results of all lemmas, theorems and corollaries 
except those in Section~\ref{sec:LargeXDistribs} are stated in terms of the functions $\mu(t; \alpha)$ and $\beta(t; \alpha)$ 
defined by Eq.~(\ref{mubetaDef}).  In each case corresponding results for a critical diffusion are obtained by substituting $2/t$ and $t/2$ respectively.  
Without loss of generality, numerical simulations for sub- and supercritical Feller diffusions are stated in terms of `dimensionless' times and 
population sizes $\alpha t$ and $\alpha x$ respectively.  

%

\subsection{Number of ancestors and coalescent times}
\label{sec:numberOfAncestors}

The existence of an underlying pre-limit finite-population process gives meaning to the concept of a 
coalescent tree for a Feller diffusion.  For the remainder of the paper we assume a Feller diffusion 
$\big(X(\tau)\big)_{\tau \in [0, t]}$ and explore the coalescent tree relative to a `current'  population $X(t)$.

Define $A_n(s)$ to be the number of ancestors existing at the time $t - s$ of a uniformly and independently chosen sample of $n$ individuals taken at time $t$.  
$A_\infty(s)$ refers to ancestors of the population $X(t)$, and $A_n(t)$ is the number of ancestors of the 
sample existing at time 0.  From the interpretation of the Poisson-Gamma solution Eq.~(\ref{Poisson_Gamma}) as a sum of families, each 
descended from a single founder at time 0, we have 
\begin{equation*}	
(A_\infty(t) \mid X(0) = x_0) \sim \text{Poisson}(x_0 \mu(t)), 
\end{equation*}
and for $l = 1, 2, \ldots$, 
\[
\mathbb{P}(A_\infty(t)  = l, X(t) \in(x, x + dx) \mid X(0) = x_0) = f_l(x_0, x; t) dx, 
\] 
where $f_l(x_0, x; t)$ is defined by Eq.~(\ref{flDef}).

\citet[Lemma~1]{BurdenGriffiths24} states that for a 
general binary exchangeable tree: 
\begin{lemma} \label{lemma:probAnGivenAinf}
Conditional on $A_\infty(s) = l > 0$, the probability that a sample of size $n > 0$ taken at time $t$ has $k \le l$ 
ancestors at time $t - s$ is independent of $X(0)$ and $X(t)$ and is equal to 
\begin{equation}	\label{probAnGivenAinf}
\mathbb{P}(A_n(s) = k \mid A_\infty(s) = l) = \binom{l}{k} \frac{n!}{l_{(n)}} \binom{n - 1}{k - 1}, \qquad 1 \le k \le l, 
\end{equation}
where $l_{(n)} = l(l + 1)\cdots(l + n - 1)$ is the rising factorial.  
\end{lemma}

A generic example of a coalescent tree is shown in Fig.~\ref{fig:CoalescentTree}.  
In a coalescent tree with $n$ leaves, define $T_k^{(n)}$, $k = 1, \ldots, n$ to be the time of the $k$-ancestor to $(k - 1)$-ancestor coalescent event measured back
from the current time $t$.  $T_1^{(n)}$ is the root of the tree.  
For convenience define $T_{n + 1}^{(n)} := 0$, and $A_n(s) := 0$ for $s > T_1^{(n)}$.  In diffusion models, coalescent trees of the entire 
population ``come down from infinity'' in general.  For these trees, we suppress the superscript in $T_1^{(\infty)}$ and write the coalescent times as the sequence 
of coalescent times as $T_1 >  T_2 >  \cdots$.  The cumulative distribution function of the $k$th coalescent time is 
\begin{equation}	\label{TkFromAn}
F_{T_k}^{(n)}(s) := \mathbb{P}(T_k \le s) = \sum_{j = 0}^{k - 1} \mathbb{P}(A_n(s) = j), \qquad k = 1, \ldots, n.
\end{equation}

\begin{figure}[t]
 \centering
 \includegraphics[width=0.75\linewidth]{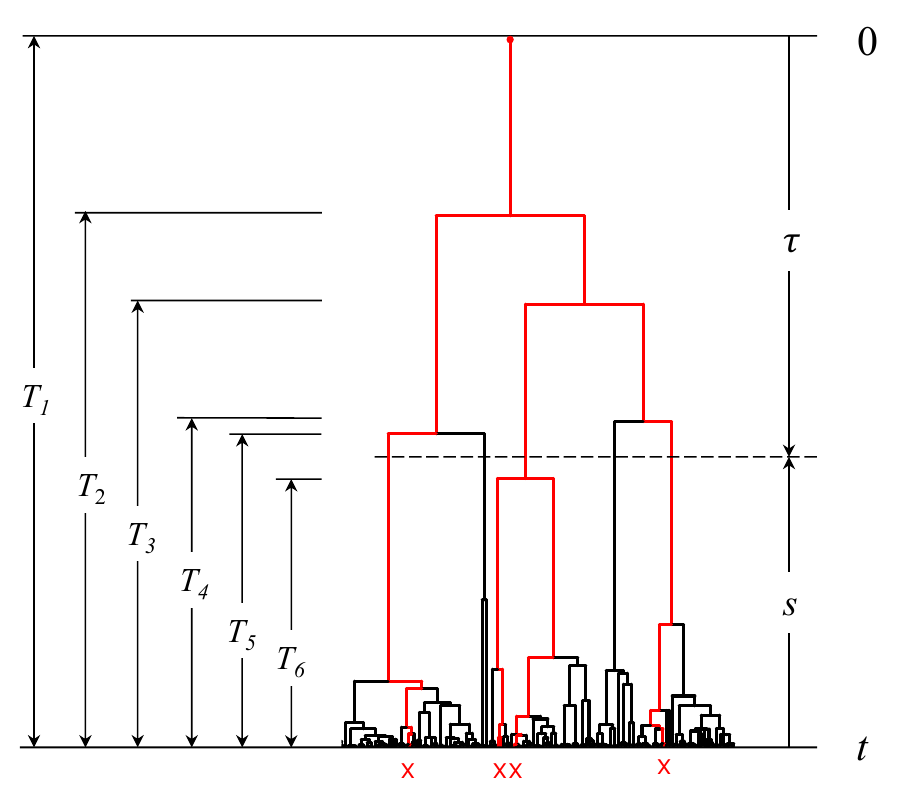}
 \caption{Coalescent tree of a Feller diffusion $\big(X(\tau)\big)_{\tau \in [0, t]}$ conditioned on a single ancestral founder of the existing current population $X(t) > 0$.  
 Coalescent times for the population are $T_1 > T_2 > \ldots$  Highlighted in red is the coalescent tree of a random sample of size $n = 4$.  In this example, 
 the number of ancestors of the sample at the time $s$ back from the present is $A_4(s) = 3$, of the population is $A_\infty(s) = 5$, and the sample coalescent times 
 are $T_2^{(4)} = T_2$, $T_3^{(4)} = T_3$ and $T_4^{(4)} = T_6$.}
 \label{fig:CoalescentTree}
\end{figure}

For the remainder of this paper we assume a diffusion conditioned on non-extinction at time $t$ and (except Section~\ref{sec:improperUnifX0}) 
a single ancestral founder of the existing current population at time zero. That is, we condition on $T_1 = t$, or equivalently, $(A_\infty(t) = 1, X(t) > 0)$.   
The equivalence follows from the Markovian property that the evolution of the process over the interval $[0, t]$ is determined solely 
by the state at time zero.  Both conditions state that at the instant the process is stopped, the final population exists and had exactly one 
ancestor at a time $t$ in the past.  

%

\section{Populations descended from a single ancestor}
\label{sec:singleAncestor}

Distributions related to populations descended from a single founding ancestor can be found 
by making use of appropriate limits as $x_0 \to 0$ to extract the $l = 1$ contribution from Eq.~(\ref{Poisson_Gamma}).  
This procedure is equivalent to setting $M_\epsilon(0) = 1$ in the pre-limit BD process or $Y(0) = 1$ in the pre-limit BGW process as 
described in Section~\ref{sec:DiffusionLimits}.  Similar limits are established practice for determining the age of a single mutant allele 
in diffusion Wright-Fisher models~\cite[Eqs.~(13) and(14)]{Kimura_1973}.  In a review paper \citet[Section~10]{BurdenGriffiths25a} 
consider the coalescent for a Feller diffusion initiated with a single founder infinitely far in the past by taking  $x_0 \to 0$.

The following lemma reproduces the findings of \citet[Theorem~2.1]{o1995genealogy}.
\begin{lemma} \label{lemma:OConnell}
For a Feller diffusion $\big(X(\tau)\big)_{\tau \in [0, t]}$ descended from a single founder at time zero and conditioned on non-extinction, the density of 
the final population $X(t)$ is
\begin{equation}	\label{expFromSingleAncestor}
f_{X(t)\mid T_1 = t}(x) = \frac{1}{\beta(t)} e^{-x/\beta(t)}, \quad x > 0.  
\end{equation}
\end{lemma}
\begin{proof}
For $x > 0$, 
\begin{eqnarray*}
\lefteqn{\mathbb{P}(X(t) \in (x, x + dx) \mid T_1 = t)} \nonumber \\
	& = & \mathbb{P}(X(t) \in (x, x + dx) \mid A_\infty(t) = 1, X(t) > 0) \nonumber \\
	& = & \lim_{x_0 \to 0} \mathbb{P}(X(t) \in (x, x + dx) \mid X(0) = x_0, X(t) > 0) \nonumber \\
	& = & \lim_{x_0 \to 0} \frac{\mathbb{P}(X(t) \in (x, x + dx) \mid X(0) = x_0)}{\mathbb{P}(X(t) > 0 \mid X(0) = x_0)} \nonumber \\
	& = & \lim_{x_0 \to 0} \frac{(f_1(x_0, x; t) + \mathcal{O}(x_0^2)) dx}{1 - e^{-x_0 \mu(t)}} \nonumber \\
	& = & \frac{1}{\beta(t)} e^{-x/\beta(t)} dx.  
\end{eqnarray*}
\end{proof}

%
%
\subsection{Distributions of ancestors of the population at an intermediate time and of coalescent times conditioned on a current population}
\label{sec:IntermediateAncestors}

The following lemma calculates the distribution of the number of ancestors $A_\infty(s)$ at an intermediate time $\tau = t - s \in [0, t]$ of a current 
population $X(t) = x > 0$ descended from a single initial ancestor. 
\begin{lemma} \label{lemma:IntermediateAncestors}
Consider a Feller diffusion $\big(X(\tau)\big)_{\tau \in [0, t]}$ with parameter $\alpha \in \mathbb{R}$,  
descended from a single founder at time zero and conditioned on a current population $X(t) = x$. 
Then for $x > 0$ and $s \in [0, t]$, 
\begin{eqnarray}	\label{ancestorsGivenT1}
\lefteqn{\mathbb{P}(A_\infty(s) = k \mid T_1 = t, X(t) = x)} \nonumber \\
 	& = & \frac{(x/\beta(s; |\alpha|) - x/\beta(t; |\alpha|))^{k - 1}}{(k - 1)!} e^{-(x/\beta(s; |\alpha|) - x/\beta(t; |\alpha|))}, \quad k = 1, 2, \ldots
\end{eqnarray}
\end{lemma}
\begin{proof}
The required probability is 
\begin{eqnarray*}
\lefteqn{\mathbb{P}(A_\infty(s) = k \mid A_\infty(t) = 1, X(t) = x > 0)} \nonumber \\
	& = & \lim_{x_0 \to 0} \frac{\mathbb{P}(A_\infty(s) = k, X(t) \in (x, x + dx) \mid X(0) = x_0)}
							{\mathbb{P}(X(t) \in (x, x + dx) \mid X(0) = x_0))}  \nonumber \\
\end{eqnarray*}
The pre-limit numerator is, up to a factor $dx$, 
\begin{eqnarray*}
	& & \int_0^\infty f(x_0, z; t - s) f_k(z, x; s) dz \nonumber \\ 
	& = & \int_0^\infty \left(x_0  \frac{\mu(t - s)}{\beta(t - s)} e^{-z/\beta(t - s)} + \mathcal{O}(x_0^2)\right) \times 
										\frac{(z\mu(s))^k}{k!} e^{-z\mu(s)}\frac{x^{k - 1}}{(k - 1)! \beta(s)^k} e^{-x/\beta(s)} dz \nonumber \\
	& = & \frac{x_0 \mu(t - s)}{\beta(t - s)} \frac{x^{k - 1}}{(k - 1)! \beta(s)^k} e^{-x/\beta(s)} \int_0^\infty \frac{(z\mu(s))^k}{k!} e^{-zU(t, s)} dz + \mathcal{O}(x_0^2) \nonumber \\
	& = & \frac{x_0 \mu(t - s)}{\beta(t - s)} \frac{x^{k - 1}}{(k - 1)! \beta(s)^k} e^{-x/\beta(s)} \frac{\mu(s)^k}{U(t, s)^{k + 1}} + \mathcal{O}(x_0^2) \\
	& = & \frac{x_0}{(k - 1)!} \frac{\mu(s)}{\beta(s)}\frac{1}{U(t, s)^2}\frac{\mu(t - s)}{\beta(t - s)} 
										\left(x\frac{\mu(s)}{\beta(s)U(t, s)}\right)^{k - 1} e^{-x/\beta(s)} + \mathcal{O}(x_0^2) \nonumber\\
	& = & \frac{x_0}{(k - 1)!} \frac{\mu(t)}{\beta(t)}\left(\frac{x}{\beta(s)} - \frac{x}{\beta(t)}\right)^{k - 1}  e^{-x/\beta(s)} + \mathcal{O}(x_0^2), 
\end{eqnarray*}
where $U(t, s)$ is defined by Eq.~(\ref{UtsDefn}) and use has been made of Eqs.~(\ref{UIdentity1}) and (\ref{UIdentity2}).  The pre-limit denominator is, up to a factor $dx$,   
\[
f(x_0, x; t) = x_0  \frac{\mu(t)}{\beta(t)} e^{-x/\beta(t)} + \mathcal{O}(x_0^2).  
\]
The result follows, taking into account the symmetry with respect to a change of sign in $\alpha$ in Eq.~(\ref{mubetaIdentity4}).  
\end{proof}

Eq.~(\ref{ancestorsGivenT1}) is a shifted Poisson distribution, which, in the following theorem, enables an interpretation of the 
coalescent times $T_2, T_3, \ldots$ as points in a non-homogeneous Poisson process with a known joint distribution.  

\begin{theorem}	\label{theorem:jointCoalescentDistrib}
Consider a Feller diffusion $\big(X(\tau)\big)_{\tau \in [0, t]}$ with parameter $\alpha \in \mathbb{R}$, 
descended from a single founder at time zero and conditioned on a current population $X(t) = x$. 
The marginal density of the population coalescent time $T_k$ is 
\begin{eqnarray}	\label{densityTkGivenx}
\lefteqn{f_{T_k \mid T_1 = t, X(t) = x}(s) = \frac{(x/\beta(s; |\alpha|) - x/\beta(t; |\alpha|))^{k - 2}}{(k - 2)!} \frac{x \mu(s; |\alpha|)}{2\beta(s; |\alpha|)} e^{-(x/\beta(s; |\alpha|) - x/\beta(t; |\alpha|))},} \qquad\qquad\qquad\nonumber \\
					& & \qquad\qquad\qquad\qquad\qquad\qquad\qquad t > s > 0,\; k = 2, 3, \ldots
\end{eqnarray}
and the joint density of the $k - 1$ population coalescent times $T_2, \ldots, T_k$ is 
\begin{eqnarray}	\label{jointTkGivenx}
f_{T_2, \ldots, T_k \mid T_1 = t, X(t) = x}(s_2, \ldots, s_k) & = & \left(\prod_{j = 2}^k \frac{x \mu(s_j; |\alpha|)}{2\beta(s_j; |\alpha|)} \right) 
									e^{-(x/\beta(s_k; |\alpha|) - x/\beta(t; |\alpha|))}, \nonumber \\
	&  & \qquad\qquad t > s_2 > s_3 > \ldots > s_k > 0. 
\end{eqnarray}
\end{theorem}
\begin{proof}
Recall that for non-homogeneous Poisson process $\big(N(\tau)\big)_{\tau > 0}$ with instantaneous rate $\lambda(\tau)$, the number of events 
$N(\tau)$ occurring in the interval $(0, \tau)$ is Poisson distributed with mean $\Lambda(\tau) := \int_0^\tau \lambda(\xi) d\xi$, 
that the marginal distribution of $U_r$ is
\[
f_{U_r}(\tau) = \frac{\Lambda(\tau)^{r - 1}}{(r - 1)!} \lambda(\tau)e^{\Lambda(\tau)}, 
\]
and that the joint density of the first $r$ event times $U_1, U_2, \ldots, U_r$ is 
\[
f_{U_1, \ldots, U_r}(\tau_1, \ldots, \tau_r) = \left(\prod_{j = 1}^r \lambda(\tau_j) \right) e^{-\Lambda(\tau_r)}, \qquad \tau_1 < \tau_2 < \ldots < \tau_r. 
\]
Consistent with Eq.~(\ref{ancestorsGivenT1}), identify $\tau \in (0, t)$ with $t - s$ and $N(\tau)$ with $A_\infty(s) - 1$ and 
\[
\Lambda(\tau) = \frac{x}{\beta(t - \tau)} - \frac{x}{\beta(t)}.  
\]
Then it is clear from Fig.~\ref{fig:CoalescentTree} that $U_j = t - T_{j + 1}$, $j = 2, 3, \ldots$ and hence 
\[
f_{T_k \mid T_1 = t, X(t) = x}(s) = \frac{(x/\beta(s) - x/\beta(t))^{k - 2}}{(k - 2)!} 
\left(\frac{d}{ds}\left\{\frac{-x}{\beta(s)}\right\} \right) e^{-(x/\beta(s) - x/\beta(t))},
\]
and 
\begin{eqnarray*}
f_{T_2, \ldots, T_k \mid T_1 = t, X(t) = x}(s_2, \ldots, s_k) & = & \left(\prod_{j = 2}^k \frac{d}{ds_j} \left\{\frac{-x}{\beta(s_j)}\right\} \right) e^{-(x/\beta(s_k) - x/\beta(t))} \nonumber \\
	&  & \qquad\qquad\qquad\qquad s_2 > s_3 > \ldots > s_r > 0. 
\end{eqnarray*}
Eqs.~(\ref{densityTkGivenx}) and (\ref{jointTkGivenx}) then follow from Eq.~(\ref{mubetaIdentity1}) and the symmetry with respect to a change of 
sign in $\alpha$ in Eq.~(\ref{mubetaIdentity4}).  
\end{proof}

%
%
\subsection{Distributions of ancestors of the population without conditioning on $X(t)$}
\label{sec:IntermediateAncestors}

The distributions of the number of ancestors at $t - s$ and coalescent times {\em without} conditioning on $X(t) = x$ but conditioning on non-extinction 
at time $t$ can be calculated in a similar fashion.  Alternatively one can start directly from Eqs.~(\ref{ancestorsGivenT1}), (\ref{expFromSingleAncestor}) 
and (\ref{jointTkGivenx}) as follows.  
\begin{theorem}	\label{theorem:withoutConditioningX}
For a Feller diffusion $\big(X(\tau)\big)_{\tau \in [0, t]}$ descended from a single founder at time zero and conditioned on non-extinction at time $t$, 
the probability that there are $k$ ancestors at time $t - s$ for $s \in [0, t]$ is
\[
 \mathbb{P}(A_\infty(s) = k \mid T_1 = t, X(t) > 0) = \frac{\beta(s)}{\beta(t)} \left(1 - \frac{\beta(s)}{\beta(t)}\right)^{k - 1}, \qquad k = 1, 2, \ldots;  
\]
the marginal density of the population coalescent time $T_k$ is  
\[
f_{T_k \mid T_1 = t, X(t) >0}(s) = \tfrac{1}{2}(k - 1)\frac{\mu(s)\beta(s)}{\beta(t)} \left(1 - \frac{\beta(s)}{\beta(t)}\right)^{k - 2}, 
					\quad s > 0,\; k = 2, 3, \ldots ;
\]
and the joint density of the $k - 1$ population coalescent times $T_2, \ldots, T_k$ is 
\[
f_{T_2, \ldots, T_k \mid T_1 = t, X(t) > 0}(s_2, \ldots, s_k) = (k - 1)! \frac{\beta(s_k)^k}{\beta(t)} \prod_{j = 2}^k \frac{\mu(s_j)}{2\beta(s_j)} 
														\qquad s_2 > s_3 > \ldots > s_k > 0.   
\]
\end{theorem}
\begin{proof}
From Lemmas~\ref{lemma:OConnell} and \ref{lemma:IntermediateAncestors}, 
\begin{eqnarray*}
\lefteqn{\mathbb{P}(A_\infty(s) = k \mid T_1 = t, X(t) > 0)} \nonumber \\
	& = & \int_{x = 0+}^\infty \mathbb{P}(A_\infty(s) = k, X(t) \in (x, x + dx) \mid T_1 = t)  \nonumber \\
	& = & \int_{x = 0+}^\infty \mathbb{P}(A_\infty(s) = k, \mid T_1 = t, X(t) = x))  \nonumber \\
	& & \qquad\qquad\qquad\qquad  \times\, \mathbb{P}(X(t) \in (x, x + dx) \mid T_1 = t)  \nonumber \\
	& = & \int_0^\infty \frac{(x/\beta(s) - x/\beta(t))^{k - 1}}{(k - 1)!} e^{-(x/\beta(s) - x/\beta(t))} \frac{1}{\beta(t)} e^{-x/\beta(t)} dx   \nonumber \\
	& = & \frac{\beta(s)}{\beta(t)} \left(1 - \frac{\beta(s)}{\beta(t)}\right)^{k - 1}, \qquad k = 1, 2, \ldots.  
\end{eqnarray*}
Similarly, from Theorem~\ref{theorem:jointCoalescentDistrib} 
\begin{eqnarray*}
\lefteqn{f_{T_k \mid T_1 = t, X(t) >0}(s)} \nonumber \\
	& = & \int_{0+}^\infty f_{T_k \mid T_1 = t, X(t) = x}(s) \times \mathbb{P}(X(t) \in (x, x + dx) \mid T_1 = t)  \nonumber \\
	& = & \int_0^\infty \frac{(x/\beta(s) - x/\beta(t))^{k - 2}}{(k - 2)!} \frac{x \mu(s)}{2\beta(s)} e^{-(x/\beta(s) - x/\beta(t))} 
					\times\, \frac{1}{\beta(t)} e^{-x/\beta(t)} dx   \nonumber \\ 
	& = & \frac{(1/\beta(s) - 1/\beta(t))^{k - 2}}{(k - 2)!} \frac{\mu(s)}{2\beta(s)\beta(t)} \int_0^\infty x^{k - 1}  e^{-x/\beta(s)} dx   \nonumber \\
	& = & \tfrac{1}{2}(k - 1)\frac{\mu(s)\beta(s)}{\beta(t)} \left(1 - \frac{\beta(s)}{\beta(t)}\right)^{k - 2},  
\end{eqnarray*}
and 
\begin{eqnarray*}
\lefteqn{f_{T_2, \ldots, T_k \mid T_1 = t, X(t) > 0}(s_2, \ldots, s_k) } \nonumber \\
	& = & \int_0^\infty  \left(\prod_{j = 2}^k \frac{x \mu(s_j)}{2\beta(s_j)} \right) e^{-(x/\beta(s_k) - x/\beta(t))} \times \frac{1}{\beta(t)} e^{-x/\beta(t)} dx   \nonumber \\
	& = & \frac{1}{\beta(t)} \left(\prod_{j = 2}^k \frac{\mu(s_j)}{2\beta(s_j)} \right) \int_0^\infty x^{k - 1} e^{-x/\beta(s_k)} dx   \nonumber \\ 
	& = & (k - 1)! \frac{\beta(s_k)^k}{\beta(t)} \prod_{j = 2}^k \frac{\mu(s_j)}{2\beta(s_j)} \qquad s_2 > s_3 > \ldots > s_k > 0.   
\end{eqnarray*}
\end{proof}

%
%
\subsection{Sampling distributions of $A_k(s)$ and $T_k$}
\label{sec:SamplingDistributions}

The distribution of the number of ancestors of a finite i.i.d.\ sample of a current population $X(t) = x$ at an intermediate time $t - s$ can be 
derived with the help of Lemma~\ref{lemma:probAnGivenAinf}.  A critical step in the proof relies on the fact that 
Eq.~(\ref{probAnGivenAinf}) is independent of $X(0)$ and $X(t)$.  
\begin{theorem}	\label{theorem:AnSample}
Consider a Feller diffusion $\big(X(\tau)\big)_{\tau \in [0, t]}$ with parameter $\alpha \in \mathbb{R}$, 
descended from a single founder at time zero and conditioned on a current population $X(t) = x$. 
Then for $x > 0$ and $s \in [0, t]$, the number $A_n(s)$ of ancestors of a finite i.i.d.\ sample of the current population is distributed as 
\begin{eqnarray*}	
\lefteqn{\mathbb{P}(A_n(s) = k \mid T_1 = t, X(t) = x)} \nonumber \\ 
	& = & \frac{n_{[k]}}{n_{(k)} (k - 1)!}(x\eta(s, t))^{k - 1} e^{-x\eta(s, t)}~_1F_1\left(k + 1, k + n, x\eta(s, t)\right),  \nonumber \\
	& & \qquad\qquad\qquad\qquad\qquad\qquad k = 1, \ldots, n,
\end{eqnarray*}
where $ \eta(s, t) = 1/\beta(s; |\alpha|)-1/\beta(t; |\alpha|)$, $n_{(k)} = n(n + 1)\cdots(n + k - 1)$ is the rising factorial, $n_{[k]} = n(n - 1)\cdots(n - k +1)$ is the falling factorial, 
and $_1F_1$ is Kummer's confluent hypergeometric function~\citep[Eq.~(13.1.2)]{Abramowitz:1965sf} 
\[
~_1F_1(a; b; z) \equiv M(a,b,z) :=\sum_{j=0}^\infty\frac{a_{(j)}}{b_{(j)}}\frac{z^j}{j!}.
\]
\end{theorem}
\begin{proof}
From Lemmas~\ref{lemma:probAnGivenAinf} and~\ref{lemma:IntermediateAncestors},
\begin{eqnarray*}	
\lefteqn{\mathbb{P}(A_n(s) = k \mid T_1 = t, X(t) = x)} \nonumber \\ 
	& = & \lim_{x_0 \to 0} \mathbb{P}(A_n(s) = k \mid X(0) = x_0, X(t) = x) \nonumber \\
	& = & \sum_{l = k}^\infty \lim_{x_0 \to 0} \mathbb{P}(A_n(s) = k \mid A_\infty(s) = l, X(0) = x_0, X(t) = x) \times \nonumber \\
	&    & \qquad\qquad\qquad \mathbb{P}(A_\infty(s) = l \mid X(0) = x_0, X(t) = x) \nonumber \\ 
	& = & \sum_{l = k}^\infty\mathbb{P}(A_n(s) = k \mid A_\infty(s) = l) \times \mathbb{P}(A_\infty(s) = l \mid T_1 = t, X(t) = x) \nonumber \\ 	
	& = & \sum_{l = k}^\infty \binom{l}{k} \frac{n!}{l_{(n)}} \binom{n - 1}{k - 1} \times \frac{1}{(l - 1)!} (x\eta(s, t))^{l - 1} e^{-x\eta(s, t)} \nonumber \\ 
	& = & \frac{n!}{k!} \binom{n - 1}{k - 1} \sum_{l = k}^\infty \frac{l!}{(l - k)!} \frac{1}{(l + n - 1)!} (x\eta(s, t))^{l - 1} e^{-x\eta(s, t)} \nonumber \\ 
	& = & \frac{n_{[k]}}{n_{(k)} (k - 1)!} (x\eta(s, t))^{k - 1} 
								\sum_{j = 0}^\infty \frac{(k + 1)_{(j)}}{(k + n)_{(j)}} \frac{1}{j!} (x\eta(s, t))^j e^{-x\eta(s, t)} \nonumber \\ 
	& = & \frac{n_{[k]}}{n_{(k)} (k - 1)!} (x\eta(s, t))^{k - 1} e^{-x\eta(s, t)} ~_1F_1\left(k + 1, k + n, x\eta(s, t)\right),  \quad k = 1, \ldots, n,  \nonumber
\end{eqnarray*}
\end{proof}

In the following theorem the distribution of the number of ancestors of a finite i.i.d.\ sample of size $n$ {\em without} conditioning on $X(t) = x$ 
is found in a similar fashion.  
\begin{theorem}
Consider a Feller diffusion $\big(X(\tau)\big)_{\tau \in [0, t]}$ descended from a single founder at time zero. 
For $s \in [0, t]$, the number $A_n(s)$ of ancestors of a finite i.i.d.\ sample of size $n > 0$ of the current population is distributed as 
\begin{equation}	\label{sampledAncestors}
\mathbb{P}(A_n(s) = k \mid T_1 = t) 
			= \sum_{l = k}^\infty  \binom{l}{k} \frac{n!}{l_{(n)}} \binom{n - 1}{k - 1}  \frac{\beta(s)}{\beta(t)} \left(1 - \frac{\beta(s)}{\beta(t)}\right)^{l - 1}. 
\end{equation}
The cumulative distribution function of the time $T_2^{(n)}$ back to the MRCA of the sample is 
\begin{equation*}	
\mathbb{P}(T_2^{(n)} < s\mid T_1 = t) = 
				\frac{\beta(s)}{\beta(t)}~_2F_1\left(2, 1; n + 1; 1 - \frac{\beta(s)}{\beta(t)}\right), \quad 0 \le s \le t, 
\end{equation*}	
where 
\[
_2F_1(a, b; c; z) = \sum_{j=0}^\infty\frac{a_{(j)}b_{(j)}}{c_{(j)}}\frac{z^j}{j!}
\]
is the ordinary hypergeometric function \citep[Eq.~(15.1.5)]{Abramowitz:1965sf}.  An alternative form for distribution function of the MRCA of 
a sample of size $n = 2$ is 
\begin{equation}	\label{MRCAnEquals2}	
\mathbb{P}(T_2^{(2)} < s\mid T_1 = t) = 
			2\frac{\beta(s)}{\beta(s) - \beta(t)} + 2\frac{\beta(s)\beta(t)}{(\beta(s) - \beta(t))^2} \log\frac{\beta(t)}{\beta(s)}, \quad 0 \le s \le t.  
\end{equation}	
\end{theorem}
\begin{proof}
Eq.~(\ref{sampledAncestors}) follows directly from Lemma~\ref{lemma:probAnGivenAinf} and Theorem~\ref{theorem:withoutConditioningX}.  
For the MRCA, setting $k = 1$ in Eq.~(\ref{sampledAncestors}), 
\begin{eqnarray*}	
\mathbb{P}(T_2^{(n)} < s\mid T_1 = t) & = & \mathbb{P}(A_n(s, t) = 1 \mid  T_1 = t) \nonumber \\
	& = & n! \frac{\beta(s)}{\beta(t)} \sum_{l = 1}^\infty \frac{l!}{(l + n - 1)!} \left(1 - \frac{\beta(s)}{\beta(t)}\right)^{l - 1} \nonumber \\
	& = & \frac{\beta(s)}{\beta(t)}~_2F_1\left(2, 1; n + 1; 1 - \frac{\beta(s)}{\beta(t)}\right), \quad 0 \le s \le t, 
\end{eqnarray*}	
Alternatively, setting $n = 2$ in the second line above, 
\begin{eqnarray*}	
\mathbb{P}(T_2^{(2)} < s\mid T_1 = t) & = & 2 \frac{\beta(s)}{\beta(t)} \sum_{l = 1}^\infty \frac{1}{(l + 1)} \left(1 - \frac{\beta(s)}{\beta(t)}\right)^{l - 1} \nonumber \\
	& = & 2 \frac{\beta(s)}{\beta(t)}\left(\frac{\beta(t) - \beta(s)}{\beta(t)}\right)^{-2} 
				\left\{\sum_{j = 1}^\infty \frac{1}{j}   \left(1 - \frac{\beta(s)}{\beta(t)}\right)^j - \frac{\beta(t) - \beta(s)}{\beta(t)} \right\}  \nonumber \\
	& = & 2\frac{\beta(s)}{\beta(s) - \beta(t)} + 2\frac{\beta(s)\beta(t)}{(\beta(s) - \beta(t))^2} \log\frac{\beta(t)}{\beta(s)}, \quad 0 \le s \le t.  
\end{eqnarray*}	
\end{proof}
Eq.~(\ref{MRCAnEquals2}) was found previously using different methods by \citet[Theorem~2.3 with $x = 1$]{o1995genealogy} and \citet[p1369]{Harris20}.  
\ref{sec:connectionOC_HJR} explains in detail the equivalence of the near-critical limit used in these two papers and the diffusion limit employed in the current paper.  

%
%
\subsection{Density of the population size at an intermediate time}

The following theorem pertains to the population existing at an intermediate time $t - s$, including lineages which subsequently become extinct 
before the current time~$t$.  
\begin{theorem}	\label{theorem:pastPopn}
Consider a Feller diffusion $\big(X(\tau)\big)_{\tau \in [0, t]}$ with parameter $\alpha \in \mathbb{R}$, 
descended from a single founder at time zero and conditioned on a current population $X(t) = x > 0$. 
The population size density at an intermediate time $t - s$ for $0< s < t$ is 
\[
f_{X(t - s) \mid T_1=t, X(t) = x} (z) 
	=  \frac{\mu(t - s; |\alpha|)}{\beta(t - s; |\alpha|)}\frac{\beta(t; |\alpha|)}{\mu(t; |\alpha|)} e^{-z/\beta(t - s; |\alpha|)} e^{x/\beta(t; |\alpha|)} f(z, x; s, |\alpha|), \qquad z >0,  
\]
where $f(z, x; s)$ is defined by Eq.~(\ref{Poisson_Gamma}). 
\end{theorem}
\begin{proof}
\begin{eqnarray*}
\lefteqn{\mathbb{P}(X(t - s) \in (z, z + dz) \mid T_1=t, X(t) = x)} \nonumber \\
	& = & \lim_{x_0 \to 0} \mathbb{P}(X(t - s) \in (z, z + dz) \mid X(0) = x_0, X(t) = x) \nonumber \\
	& = & \lim_{x_0 \to 0} \frac{\mathbb{P}(X(t - s) \in (z, z + dz), X(t) \in (x, x + dx) \mid X(0) = x_0)}
									{\mathbb{P}(X(t) \in (x, x + dx)\mid X(0) = x_0)} \nonumber \\
	& = & \lim_{x_0 \to 0} \frac{f(x_0, z; t - s)dz\, f(z, x; s)dx}{f(x_0, x; t)dx} \nonumber \\  
	& = & \lim_{x_0 \to 0} \frac{(x_0 \mu(t - s) \beta(t - s)^{-1} e^{-z/\beta(t - s)} + \mathcal{O}(x_0^2)) f(z, x; s)}
							{x_0 \mu(t) \beta(t)^{-1} e^{-x/\beta(t)} + \mathcal{O}(x_0^2)} dz  \nonumber \\
	& = & \frac{\mu(t - s)}{\beta(t - s)}\frac{\beta(t)}{\mu(t)} e^{-z/\beta(t - s)} e^{x/\beta(t)} f(z, x; s) dz. 
\end{eqnarray*}
It is straightforward to check from Eqs.~(\ref{Poisson_Gamma}), (\ref{mubetaIdentity4}) and (\ref{UIdentity3}) that the final line is symmetric with respect 
to a change of sign in $\alpha$, and the result follows.  
\end{proof}

%
%
\section{$T_1 \to \infty$ limit}	\label{sec:TInfLimit}

Since the ancestral line of any gene stretches well into the past, it is useful to consider the $t \to \infty$ limit of distributions 
populations descended from a single ancestral founder at a time $T_1 = t$ in the past.  
\begin{theorem}	\label{theorem:infT1Distribs}
Consider a Feller diffusion $\big(X(\tau)\big)_{\tau \in [0, t]}$ with parameter $\alpha \in \mathbb{R}$ descended from a single founder at time zero and conditioned on a current population $X(t) = x > 0$. Define a random variable $M_n(s; x)$ by the weak limit 
\[	
M_n(s; x) := \lim_{t \to \infty} A_n(s)|(T_1 = t, X(t) = x).  
\]
For the number of ancestors of the entire population 
\begin{equation}	\label{numberOfAncestors}
\mathbb{P}(M_\infty(s; x) = k) = \frac{1}{(k - 1)!} \left(\frac{x}{\beta(s; |\alpha|)}\right)^{k - 1} e^{-x/\beta(s; |\alpha|)}, \qquad k = 1, 2, \ldots,  
\end{equation}
and for the number of ancestors of an i.i.d.\ sample
\begin{eqnarray*}	\label{nAncestorsSample}
\lefteqn{\mathbb{P}(M_n(s; x) = k)} \nonumber \\
	& = & \frac{n_{[k]}}{n_{(k)} (k - 1)!}\left(\frac{x}{\beta(s; |\alpha|)}\right)^{k - 1} e^{-x/\beta(s; |\alpha|)}~_1F_1\left(k + 1, k + n, \frac{x}{\beta(s; |\alpha|)}\right), \nonumber \\
	& & 			\qquad\qquad\qquad\qquad\qquad\qquad\qquad\qquad\qquad\qquad\qquad\qquad	k = 1, \ldots, n.
\end{eqnarray*}
Densities corresponding to the weak limit as $t \to \infty$ of the population coalescent times $T_k$ are 
\begin{eqnarray}	\label{fTkInfT1}
f_{T_k}^{{\rm inf}\,T_1}(s; x) & := & \lim_{t \to \infty}f_{T_k \mid T_1 = t, X(t) = x}(s) \nonumber \\
	& = & \frac{1}{(k - 2)!}\left(\frac{x}{\beta(s; |\alpha|)}\right)^{k - 2} \frac{x \mu(s, |\alpha|)}{2\beta(s, |\alpha|)} e^{-(x/\beta(s, |\alpha|))}, 
					\quad s > 0,\; k = 2, 3, \ldots \nonumber \\
\end{eqnarray}
and
\begin{eqnarray}	\label{joint_fTkInfT1}
f_{T_2, \ldots, T_k}^{{\rm inf}\,T_1}(s_2, \ldots, s_k; x) & := & \lim_{t \to \infty}f_{T_2, \ldots, T_k \mid T_1 = t, X(t) = x}(s_2, \ldots, s_k) \nonumber \\ 
	& = & \left(\prod_{j = 2}^k \frac{x \mu(s_j, |\alpha|)}{2\beta(s_j, |\alpha|)} \right) e^{-(x/\beta(s_k, |\alpha|))},  \qquad s_2 > s_3 > \ldots > s_k > 0.  \nonumber \\
\end{eqnarray}
The density corresponding to the weak limit as $t \to \infty$ of the population $X(t - s)$ at a time $s$ in the past is 
\begin{equation}	\label{pastDensity}
f_{X(T_1 - s)}^{{\rm inf}\,T_1}(z; x) := \lim_{t \to \infty}f_{X(t - s) \mid T_1 = t, X(t) = x} (z) =  e^{|\alpha| s} f(z, x; s; |\alpha|), \quad z > 0.  
\end{equation}
\end{theorem}
\begin{proof}
For $\alpha \in \mathbb{R}$ , $\mu(t; |\alpha|) \to 2|\alpha|$ and $\beta(t; |\alpha|)^{-1} \to 0$ as $t \to \infty$.  The first four results then follow immediately from Lemma~\ref{lemma:IntermediateAncestors}, Theorem~\ref{theorem:AnSample}, 
and the two parts of Theorem~\ref{theorem:jointCoalescentDistrib} respectively.  

For $\alpha \in \mathbb{R}$,
\[
\frac{\mu(t - s; |\alpha|)}{\mu(t; |\alpha|)} \to 1, \quad \frac{\beta(t; |\alpha|)}{\beta(t - s; |\alpha|)} \to e^{|\alpha| s}, 
\]
and the final identity then follows immediately from Theorem~\ref{theorem:pastPopn}.  
\end{proof}

\begin{figure}
 \centering
    \includegraphics[width=0.65\linewidth]{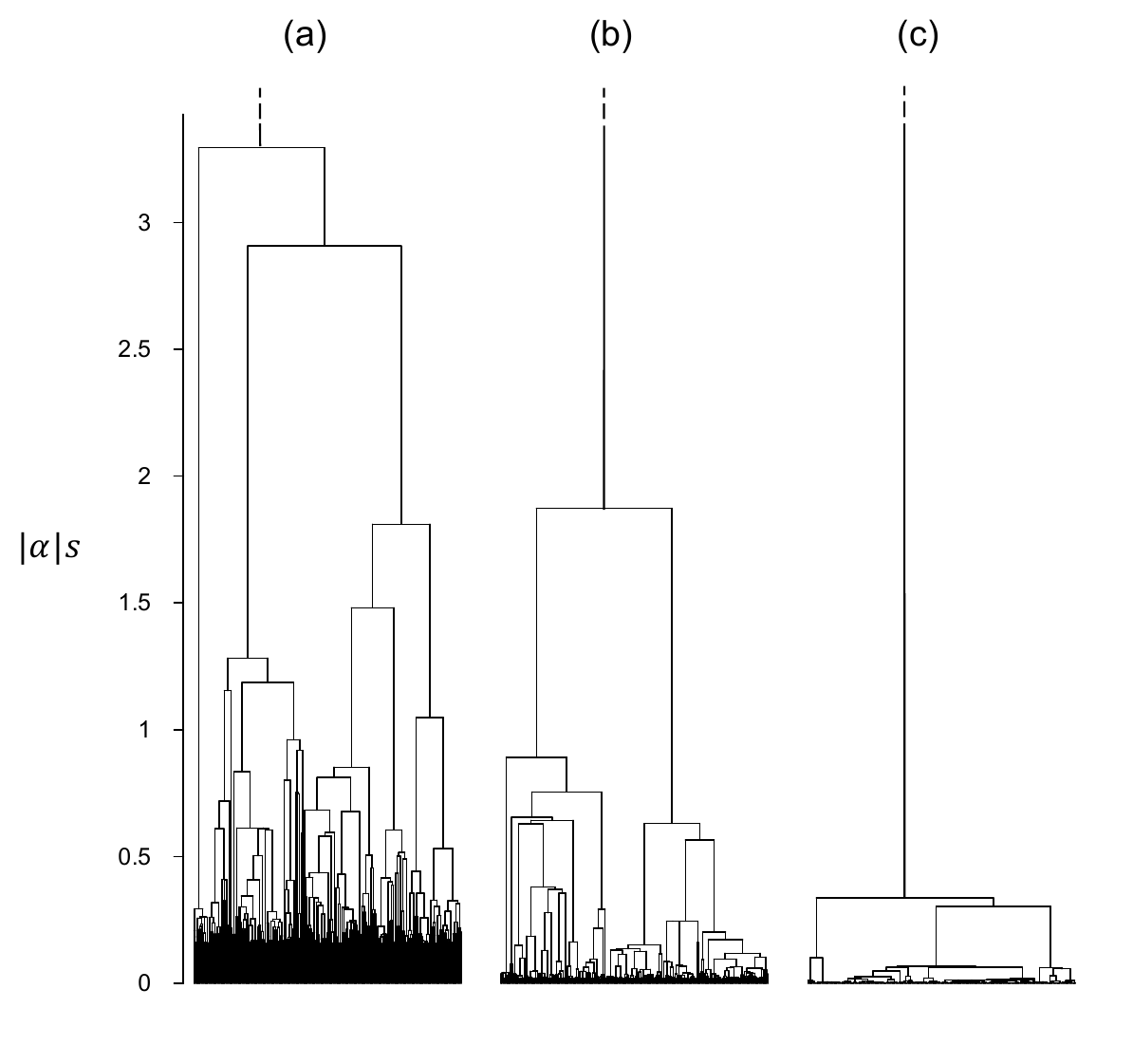}
 \caption{Simulated coalescent trees for Feller diffusions initiated from a single ancestral founder infinitely far in 
 the past conditioned on final observed populations $X(t) = x$ for (a) $\alpha x=10$, (b) $\alpha x = 1$ and (c) $\alpha x = 0.1$.  The time $s$ is measured back from the present.} 
 \label{fig:InfT1Trees}
 \end{figure}

Figure~\ref{fig:InfT1Trees} shows simulated coalescent trees computed from the joint distribution Eq.~(\ref{joint_fTkInfT1}) of coalescent times 
for a range of scaled final populations $|\alpha| x$.  \citet[Lemma~2]{BurdenGriffiths24} show that the part of these trees above a height $s = 2\epsilon$ 
can also be generated by a BD process with the birth and death rates of Eq.~(\ref{BDlimit}) taking the particular form 
\[
\hat\lambda(\epsilon) = \frac{\alpha}{1 - e^{-2\alpha\epsilon}}, \quad \hat\mu(\epsilon) = \frac{\alpha e^{-2\alpha\epsilon}}{1 - e^{-2\alpha\epsilon}}.  
\]
In Section~\ref{sec:LargeXDistribs} below we show that asymptotically the joint distribution of 
$T_1, \ldots, T_k$ for any fixed $k$ maintains its shape and is shifted upwards by an amount $\log 2\alpha x$ as $\alpha x \to \infty$.

%
%

\section{Improper Uniform prior on $T_1$}
\label{sec:UnifT1}

Various authors \citep{Aldous05,Gernhard08,Wiuf18,Ignatieva20} have analysed the coalescent tree of a finite population, continuous-time BD process conditioned on an observed final population 
(generally referred to as a {\em conditioned reconstructed process}) by imposing an improper uniform prior on the time since initiation of the process by a single founder.  
\citet{BurdenGriffiths24} have extended this coalescent tree construction to a supercritical Feller diffusion by exploiting the near-critical limit of a BD process summarised in 
Section~\ref{sec:DiffusionLimits}.  Here we reproduce defining properties of the coalescent tree, namely the distributions of numbers of ancestors $A_\infty(s)$  
and coalescent times $T_k$ for a Feller diffusion conditioned on $X(t) = x$ under the assumption of an improper uniform prior on $T_1$ directly from 
the results of Subsection~\ref{sec:IntermediateAncestors}.  

\begin{definition}
Consider a uniform prior with cutoff $\Lambda$ on the time $T_1$ since initiation from an unspecified population containing a single founding ancestor 
of the current population: 
\begin{equation*}	
f_{T_1}(t) = \begin{cases}
	\Lambda ^{-1} & 0 \le t \le \Lambda; \\
	0                     & t > \Lambda. 
	\end{cases}
\end{equation*} 
Define {\em probability subject to an improper uniform prior on $T_1$} to mean the limit of any posterior probability as $\Lambda \to \infty$.  
Such probabilities will be denoted $\mathbb{P}^{{\rm unif}\,T_1} (\cdot)$ and corresponding densities will be denoted $f^{{\rm unif}\,T_1} (\cdot)$.
\end{definition}

\begin{lemma}	\label{lemma:PosteriorT1}
For a Feller diffusion $\big(X(\tau)\big)_{\tau \in [0, T_1]}$ with parameter $\alpha \in \mathbb{R}$, 
subject to an improper uniform prior on $T_1$ and conditioned on $X(T_1) = x > 0$, the posterior 
density of $T_1$ is 
\begin{equation}		\label{uniformPosteriorT1}
f^{{\rm unif}\,T_1} _{T_1} (t; x) := f^{{\rm unif}\,T_1} _{T_1 \mid X(T_1) = x} (t) = 
	\frac{x}{2} \frac{\mu(t; |\alpha|)}{\beta(t; |\alpha|)} e^{-x/\beta(t; |\alpha|)} = \frac{d}{dt} e^{-x/\beta(t; |\alpha|)}, \quad t > 0.  
\end{equation}
\end{lemma}
\begin{proof}
Recall that when conditioning on $X(t) > 0$, the $X(0) = x_0 \to 0$ limit of the Feller diffusion solution Eq.~(\ref{Poisson_Gamma}) 
captures the density due to a single founder.  The posterior probability of $T_1$ conditioned on $X(t) = x > 0$ is then 
\begin{eqnarray*}
\lefteqn{\mathbb{P}^{{\rm unif}\,T_1} (T_1 \in (t, t + dt) \mid X(t) = x) } \\
	& = & \lim_{\Lambda \to \infty} \frac{\mathbb{P}(X(t) \in (x, x + dx) \mid T_1 = t) \times \mathbb{P}(T_1 \in (t, t + dt))}
						{\mathbb{P}(X(t) \in (x, x + dx))} \nonumber \\
	& = & \lim_{\Lambda \to \infty} \lim_{x_0 \to 0}\frac{\mathbb{P}(X(t) \in (x, x + dx) \mid X(0) = x_0) \times \Lambda^{-1} dt}
						{\int_0^\Lambda \mathbb{P}(X(\xi) \in (x, x + dx) \mid X(0) = x_0) \times \Lambda^{-1} d\xi} \nonumber \\
	& = &  \lim_{x_0 \to 0}\frac{f(x_0, x; t)dt}{\int_0^\infty f(x_0, \xi; t) d\xi}.
\end{eqnarray*}
After cancelling the factors of $x_0$ and taking the limit the numerator is 
\[
\frac{\mu(t)}{\beta(t)} e^{-x/\beta(t)} dt . 
\]
For $\alpha \ge 0$ the denominator is, from Eq.~(\ref{mubetaIdentity1}) and noting that $1/\beta(\infty) = 0$ for $\alpha \ge 0$, 
\[
\int_0^\infty \frac{\mu(\xi)}{\beta(\xi)} e^{-x/\beta(\xi)} d\xi = \frac{2}{x} \int_0^\infty \left\{\frac{d}{d\xi} e^{-x/\beta(\xi)}\right\}  d\xi = \frac{2}{x},
\]
and Eq.~(\ref{uniformPosteriorT1}) follows immediately.  
For $\alpha < 0$, $1/\beta(\infty) = -2\alpha$ and the denominator is $(2/x)\times e^{2\alpha x}$, giving (see 
Eqs.~(\ref{mubetaIdentity1}) and (\ref{mubetaIdentity3})), 
\[
f^{{\rm unif}\,T_1} _{T_1} (t; x) = \frac{x}{2} \frac{\mu(t; \alpha)}{\beta(t; \alpha)} e^{-x(\beta(t; \alpha)^{-1} + 2\alpha)} = \frac{x}{2} \frac{\mu(t; \alpha)}{\beta(t; \alpha)} e^{-x\mu(t; \alpha)} = \frac{x}{2} \frac{\mu(t; |\alpha|)}{\beta(t; |\alpha|)} e^{-x/\beta(t; |\alpha|)}
\]
as required.  
\end{proof}

The posterior distribution of the number of ancestors at a time $s$ in the past of a current population $x$ assuming an improper uniform prior 
on $T_1$ is derived in \citet[Theorem~4]{BurdenGriffiths24} for a supercritical Feller diffusion by constructing a conditioned reconstructed process 
from the limit of a continuous-time birth-death process.  \citet[Section~6]{BurdenGriffiths24} argue that their theorem does not apply to a subcritical 
diffusion because the conditioned reconstructed 
process may not converge to a single ancestor.  Following is an alternative proof, which extends this theorem to also cover subcritical diffusions
by implicitly conditioning on convergence to a single ancestor.  This conditioning is inherent in the proof of Lemma~\ref{lemma:PosteriorT1} above 
in the normalisation introduced via the application of Bayes theorem.  
\begin{theorem}	\label{theorem:unifT1Distribs}
For a Feller diffusion $\big(X(\tau)\big)_{\tau \in [0, T_1]}$with parameter $\alpha \in \mathbb{R}$ and  
subject to an improper uniform prior on $T_1$ and conditioned on $X(T_1) = x > 0$, the posterior probability 
that the current population has $k$ ancestors at time $s$ in the past is 
\begin{equation}		\label{PoisAncestors}
\mathbb{P}^{{\rm unif}\,T_1}(A_\infty(s) = k \mid X(T_1) = x) 
	= \frac{1}{k!} \left(\frac{x}{\beta(s; |\alpha|)} \right)^k e^{-x/\beta(s; |\alpha|)}, \qquad k = 0, 1, 2 \ldots
\end{equation}
The posterior marginal density of the $k$th coalescent time is 
\[
f_{T_k \mid X(T_1) = x}^{{\rm unif}\,T_1}(s) = 
	\frac{1}{(k - 1)!} \left(\frac{x}{\beta(s; |\alpha|)} \right)^{k - 1} \frac{x\mu(s; |\alpha|)}{2\beta(s; |\alpha|)} e^{-x/\beta(s; |\alpha|)}, \qquad s \ge 0,\quad k = 1, 2, \ldots,  
\]
and the joint posterior density of the $k$ population coalescent times $T_1, \ldots T_k$ is 
\begin{eqnarray*}
f_{T_1, \ldots, T_k \mid X(T_1) = x}(s_1, \ldots, s_k) & = & \left(\prod_{j = 1}^k \frac{x \mu(s_j; |\alpha|)}{2\beta(s_j; |\alpha|)} \right) e^{-x/\beta(s_k; |\alpha|)}, \nonumber \\
	& = & \left(\prod_{j = 1}^k \frac{d}{ds_j} e^{-x/\beta(s_j; |\alpha|)} \right) e^{-x/\beta(s_k; |\alpha|)}, \nonumber \\
	&  & \qquad s_1 > s_2 > \ldots > s_k > 0,\quad k = 1, 2, \ldots 
\end{eqnarray*}
\end{theorem}
\begin{proof}
From Eq.~(\ref{ancestorsGivenT1}) and (\ref{mubetaIdentity4}), for $k > 0$, 
\begin{eqnarray*}
\mathbb{P}^{{\rm unif}\,T_1}(A_\infty(s) = k \mid X(T_1) = x) 
	& = & \int_s ^\infty \mathbb{P}(A_\infty(s) = k \mid T_1 = t, X(t) = x) f^{{\rm unif}\,T_1} _{T_1 \mid X(T_1) = x} (t)  \,dt \\
	& = & e^{-x/\beta(s; |\alpha|)} \int_s^\infty  \frac{(x/\beta(s; |\alpha|) - x/\beta(t; |\alpha|))^{k - 1}}{(k - 1)!}   \left\{\frac{d}{dt} \frac{-x}{\beta(t; |\alpha|)}\right\} dt \\
	& = & e^{-x/\beta(s; |\alpha|)} \int_0^{x/\beta(s; |\alpha|)}  \frac{u^{k - 1}}{(k - 1)!} du \\
	& = & \frac{1}{k!} \left(\frac{x}{\beta(s; |\alpha|)} \right)^k e^{-x/\beta(s; |\alpha|)}, \qquad k = 1, 2 \ldots,   
\end{eqnarray*}
while for $k = 0$, 
\begin{eqnarray*}
\mathbb{P}^{{\rm unif}\,T_1}(A_\infty(s) = 0 \mid X(T_1) = x) & = & \mathbb{P}^{{\rm unif}\,T_1} (T_1 < s \mid X(T_1) = x) \\
	& = & \int_0^s f^{\rm unif}_{T_1} (\xi ; x) d\xi \\
	& = & e^{-x/\beta(s; |\alpha|)}, 
\end{eqnarray*}
which proves Eq.~(\ref{PoisAncestors}) for all $k \ge 0$.  

The marginal density of the $k$th coalescent time is, from Eqs.~(\ref{TkFromAn}) and (\ref{mubetaIdentity1}),
\begin{eqnarray*}
f_{T_k \mid X(T_1) = x}^{{\rm unif}\,T_1}(s) & = & \frac{d}{ds} \sum_{j = 0}^{k - 1} \frac{1}{j!} \left(\frac{x}{\beta(s; |\alpha|)} \right)^j e^{-x/\beta(s; |\alpha|)} \nonumber \\
	& = & \left(\sum_{j = 0}^{k - 1} \frac{1}{j!} \left(\frac{x}{\beta(s; |\alpha|)} \right)^j - \sum_{j = 1}^{k - 1} \frac{1}{(j - 1)!} \left(\frac{x}{\beta(s; |\alpha|)} \right)^{j - 1} \right)
									\frac{x\mu(s; |\alpha|)}{2\beta(s; |\alpha|)} e^{-x/\beta(s; |\alpha|)} \nonumber \\ 
	& = & \frac{1}{(k - 1)!} \left(\frac{x}{\beta(s; |\alpha|)} \right)^{k  - 1} \frac{x\mu(s; |\alpha|)}{2\beta(s; |\alpha|)} e^{-x/\beta(s; |\alpha|)}, \qquad s \ge 0.  
\end{eqnarray*}

The posterior joint density of $T_1, \ldots, T_k$ is, from Eq.~(\ref{jointTkGivenx}) and (\ref{mubetaIdentity4})  
\begin{eqnarray*}
\lefteqn{f_{T_1, \ldots, T_k \mid X(T_1) = x}^{{\rm unif}\,T_1}(s_1, \ldots, s_k)} \nonumber \\ 
	& = & f_{T_1 \mid X(T_1) = x}^{{\rm unif}\,T_1}(s_1) \times f_{T_2, \ldots, T_k \mid T_1 = s_1, X(s_1) = x}(s_2, \ldots, s_k) \nonumber \\
	& = & \frac{x\mu(s_1; |\alpha|)}{2\beta(s_1; |\alpha|)} e^{-x/\beta(s_1; |\alpha|)} 
			\times \left(\prod_{j = 2}^k \frac{x \mu(s_j; |\alpha|)}{2\beta(s_j; |\alpha|)} \right) e^{-(x/\beta(s_k; |\alpha|) - x/\beta(s_1; |\alpha|))} \nonumber \\
	& = & \left(\prod_{j = 1}^k \frac{x \mu(s_j; |\alpha|)}{2\beta(s_j; |\alpha|)} \right) e^{-x/\beta(s_k; |\alpha|)}.
\end{eqnarray*} 
The alternative form of the posterior joint density follows from Eq.~(\ref{mubetaIdentity1}).  
\end{proof}

The distributions of Theorem~\ref{theorem:unifT1Distribs} match those of Theorem~\ref{theorem:infT1Distribs} if the index $k$ is replaced by $k + 1$.  
That is, the coalescent time $T_k$ in the presence of an improper uniform on $T_1$ finds its analogue in the coalescent time $T_{k + 1}$ for a populaton descended 
from a single ancestral founder infinitely far in the past.  In terms of the simulations of Fig.~\ref{fig:InfT1Trees}, this means that each of the two subtrees below the first 
branching is itself a simulation of the coalescent tree descended from a single ancestor under an improper uniform prior.
%
%
%

\section{Improper uniform prior on $X(0)$}
\label{sec:improperUnifX0}

In this section only, we drop the assumption of a single ancestral founder and instead assume an improper uniform prior on the initial 
population $X(0)$.  
\begin{definition}
Consider a uniform prior with cutoff $\Lambda$ on the initial population $X(0)$: 
\begin{equation}	\label{uniformPriorX0}
f_{X(0)}(x_0) = \begin{cases}
	\Lambda ^{-1} & 0 \le x_0 \le \Lambda; \\
	0                     & x_0 > \Lambda, 
	\end{cases}
\end{equation}
Define {\em probability subject to an improper uniform prior on $X(0)$} to mean the limit of any posterior probability as $\Lambda \to \infty$.  
Such probabilities will be denoted $\mathbb{P}^{{\rm unif}\,X(0)} (\cdot)$ and corresponding densities will be denoted $f^{{\rm unif}\,X(0)} (\cdot)$.
\end{definition}

\begin{lemma}	\label{lemma:PosteriorX0}
For a Feller diffusion $\big(X(\tau)\big)_{\tau \in [0, t]}$ subject to an improper uniform prior on $X(0)$ and conditioned on $X(t) = x > 0$, the posterior 
density of $X(0)$ is 
\begin{equation}		\label{uniformPosteriorX0}
f^{{\rm unif}\,X(0)}_{X(0) \mid X(t) = x}(x_0) = e^{\alpha t} f(x_0, x; t), \qquad x_0 > 0.  
\end{equation}
\end{lemma}
\begin{proof}
\begin{eqnarray*}
\lefteqn{\mathbb{P}^{{\rm unif}\,X(0)}(X(0) \in (x_0, x_0 + dx_0) \mid X(t) = x)  } \nonumber \\
	& = & \lim_{\Lambda \to \infty} \frac{\mathbb{P}(X(t) \in (x, x + dx) \mid X(0) = x_0) \times \mathbb{P}(X(0) \in (x_0, x_0 + dx_0))}
						{\mathbb{P}(X(t) \in (x, x + dx))} \nonumber \\
	& = & \frac{f(x_0, x; t) dx_0}{\int_0^\infty f(\xi, x; t) d\xi },  					
\end{eqnarray*}
provided the integral in the denominator converges.  From Eq.~(\ref{Poisson_Gamma}), for $x > 0$ the denominator is 
\begin{eqnarray*}
\int_0^\infty f(\xi, x; t) d\xi & = & \int_0^\infty \sum_{l = 1}^\infty e^{-\xi \mu(t)} \frac{(\xi \mu(t))^l}{l!} \frac{x^{l - 1} e^{-x/\beta(t)}}{\beta(t)^l (l - 1)!} d\xi \\
	& = & \sum_{l = 1}^\infty \frac{1}{\mu(t) l!} \left(\int_0^\infty e^{-z} z^l dz\right)  \frac{x^{l - 1} e^{-x/\beta(t)}}{\beta(t)^l (l - 1)!} \\
	& = & \frac{1}{\mu(t)\beta(t)} \\
	& = & e^{-\alpha t},  
\end{eqnarray*}
and Eq.(\ref{uniformPosteriorX0}) follows.  
\end{proof}

The following theorem gives the posterior distribution of the number of initial ancestors.  

\begin{theorem}
For a Feller diffusion $\big(X(\tau)\big)_{\tau \in [0, t]}$ subject to an improper uniform prior on $X(0)$ and conditioned on $X(t) = x > 0$, 
the number $A_\infty(t)$ of initial ancestors of the population has the posterior probability 
\begin{equation}	\label{ancestrosUnifX0}
\mathbb{P}^{{\rm unif}\,X(0)}(A_\infty(t) = k \mid X(t) = x) = \frac{1}{(k - 1)!} \left(\frac{x}{\beta(t)}\right)^{k - 1} e^{-x/\beta(t)}, \qquad k = 1, 2, \ldots 
\end{equation}
\end{theorem}
This is a shifted Poisson distribution.  Note that the number of founders is necessarily non-zero because of the assumption that $x > 0$.  
\begin{proof}
from Eqs.~(\ref{flDef}) and~(\ref{uniformPosteriorX0}),
\begin{eqnarray*}
\lefteqn{\mathbb{P}^{{\rm unif}\,X(0)}(A_\infty(t) = k \mid X(t) = x)} \nonumber \\ 
	& = & \int_0^\infty \mathbb{P}(A_\infty(t) = k \mid X(0) = x_0, X(t) = x) \times f^{{\rm unif}\,X(0)}_{X(0) \mid X(t) = x}(x_0, t) dx_0 \nonumber \\ 
	& = & \int_0^\infty \frac{f_k(x_0, x; t)}{f(x_0, x; t)} \times e^{\alpha t} f(x_0, x; t) dx_0 \nonumber \\ 
	& = & \int_0^\infty \frac{(x_0 \mu(t))^k}{k!} e^{-x_0\mu(t)} dx_0 \, \frac{x^{k - 1}}{(k - 1)! \beta(t)^k} e^{-x/\beta(t)} e^{\alpha t} \nonumber \\ 
	& = & \frac{1}{(k - 1)!} \left(\frac{x}{\beta(t)}\right)^{k - 1} e^{-x/\beta(t)}, \qquad k = 1, 2, \ldots  
\end{eqnarray*}
\end{proof}

An alternative way to define the improper uniform prior is to replace Eq.~(\ref{uniformPriorX0}) by 
\[
f_{X(0)}(x_0) = \epsilon e^{-\epsilon x_0}, 
\]
and take the limit $\epsilon \to 0+$ after calculating the posterior distribution.  It is straightforward to check that 
\begin{equation}	\label{uniformPosteriorX0Alt}
f^{{\rm unif}\,X(0)}_{X(0) \mid X(t) = x}(x_0) = \lim_{\epsilon \to 0+}\frac{f(x_0, x; t) e^{-\epsilon x_0}}{\int_0^\infty f(\xi, x; t) e^{-\epsilon \xi} d\xi }
		= e^{\alpha t} f(x_0, x; t).  
\end{equation}
For a supercritical diffusion, this provides an alternate proof of the first and final parts of Theorem~\ref{theorem:infT1Distribs}.  
By Eq.~(\ref{expFromSingleAncestor}), conditioning on a single 
ancestor at time $t$ in the past effectively sets an exponential prior with mean $\beta(t - s)$ on the population size at time $s$ in the past.  The corresponding 
posterior distribution in the limit $t - s \to \infty$ is Eq.~(\ref{uniformPosteriorX0Alt}) with $x_0$ replaced by $z$, $t$ replaced by $s$ and the limit $\epsilon \to 0$ 
replaced by $\beta(t - s)^{-1} \to 0$ for $\alpha > 0$, which leads to Eq.~(\ref{pastDensity}).  Similarly, replacing $t$ with $s$ in Eq.~(\ref{ancestrosUnifX0}) leads 
to Eq.~(\ref{numberOfAncestors}) for $\alpha > 0$.  

%

\section{The distribution of the time and contemporaneous population size of the MRCA given $X(t) = x$}
\label{sec:MRCAgivenX}

\citet{burden2016genetic} and \citet{BurdenSoewongsoso19} set out to estimate the time back to the MRCA and population size at that time 
of a currently observed population $X(t) = x$ generated by a Feller diffusion.  Their approach, which consisted of setting the two quantities being estimated 
as initial conditions and calculating a joint likelihood function, required an additional conjecture and did not lead to fully analytic solutions.  
Details are summarised in \ref{sec:connectionBSandBS}.  The approach taken in this section avoids these shortcomings.  It 
considers a population initiated from a single ancestor at an earlier time $T_1 = t$ in the past, calculates the joint distribution of the time $T_2$ back to the MRCA and 
contemporaneous population size $X(t -T_2)$, and only then places prior assumptions on  $T_1$.  We begin with the following lemma.  

\begin{lemma}
For a Feller diffusion $\big(X(\tau)\big)_{\tau \in [0, t]}$ with parameter $\alpha \in \mathbb{R}$, 
descended from a single founder at time zero and conditioned on a current population $X(t) = x > 0$, 
the joint density of the time $T_2$ back to the MRCA of the current population and contemporaneous population size $X_{\rm MRCA} := X(t -T_2)$ is 
\begin{eqnarray}	\label{jointMRCAgivenT1}
f_{T_2, X_{\rm MRCA}}(s, z; t, x) & := & f_{T_2, X(t - T_2) \mid T_1 = t, X(t) = x}(s, z) \nonumber \\
	& = & \frac{1}{2} \frac{\mu(s; |\alpha|)^2\beta(t; |\alpha|)\mu(t - s; |\alpha|)}{\beta(s; |\alpha|)^2\mu(t; |\alpha|)\beta(t - s; |\alpha|)} 
								zxe^{-zU(t, s; |\alpha|)} e^{-x(\beta(s; |\alpha|)^{-1} - \beta(t; |\alpha|)^{-1})}, \nonumber \\
	& & 					\qquad\qquad\qquad\qquad\quad 0 < s < t; \quad z, x > 0, 
\end{eqnarray}
where $U(t, s)$ is defined by Eq.~(\ref{UtsDefn}).  
\end{lemma}
\begin{figure}[t]
 \centering
 \includegraphics[width=0.4\linewidth]{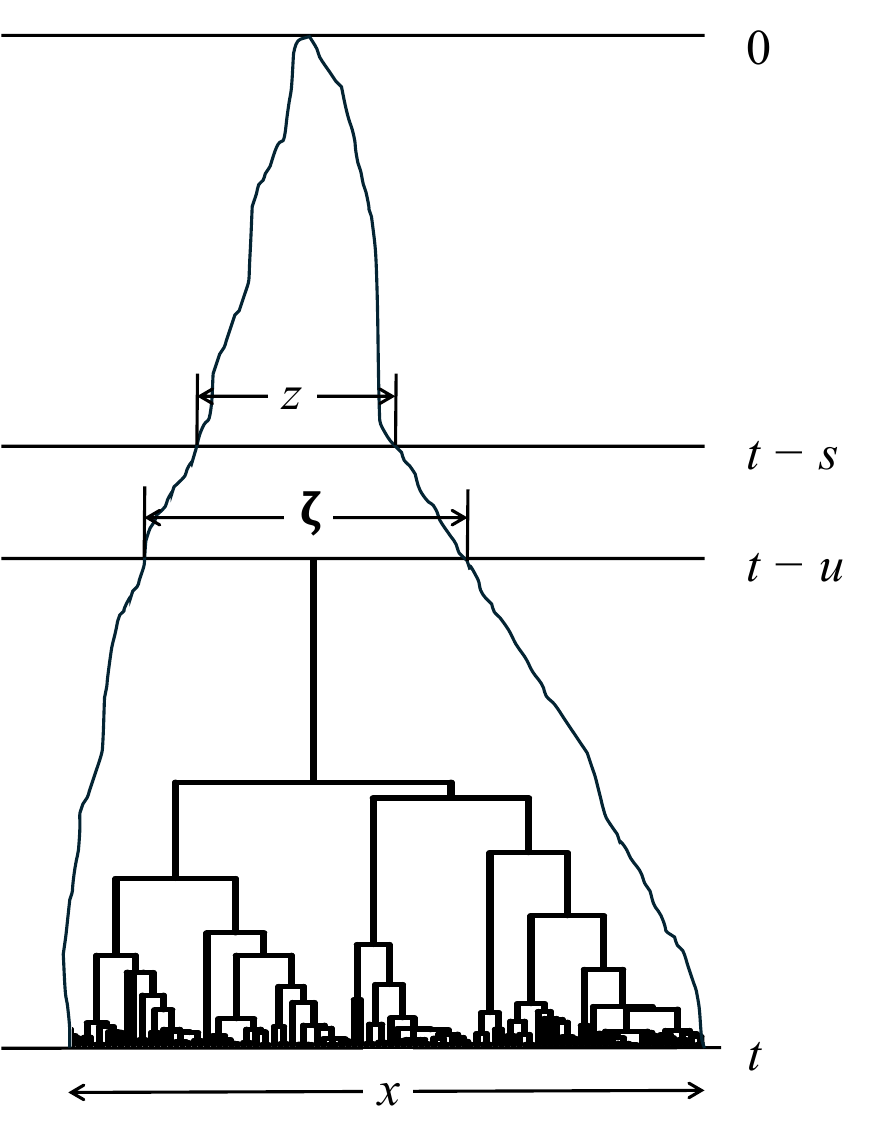}
 \caption{The function $H(t, s, u)$ is calculated by taking $X(0) = x_0 \to 0$ and integrating over the population $\zeta$ at time $t - u$.}
 \label{fig:MRCA_tree}
\end{figure}
\begin{proof}
For $0 < u < s < t$ and $x, z > 0$ define a function $H(t, s, u)$ by 
\[
H(t, s, u) dz = \mathbb{P}(X(t- s) \in (z, z + dz), A_\infty(u) = 1 \mid X(t) = x, T_1 = t).  
\]
This is the joint probability that, given a Feller process had a single founder at a time $t$ in the past and has a 
current population $x$,  the population size was $z$ at an intermediate time $s$ in the past, and at a more recent time $u$ in the past the current population 
had only one ancestor (see Fig.~\ref{fig:MRCA_tree}).  Recall that when conditioning on $X(t) > 0$, the $X(0) = x_0 \to 0$ limit of the Feller diffusion solution 
Eq.~(\ref{Poisson_Gamma}) captures the density due to a single founder.  Integrating over the population size at time $t - u$ gives,  
\[
H(t, s, u) = \lim_{x_0 \to 0} \frac{f(x_0, z; t - s) \int_0^\infty f(z, \zeta; s - u)f_1(\zeta, x, u) d\zeta}{f(x_0, x; t)}. 
\]
The $x_0$-dependent part is 
\begin{eqnarray*}
\lim_{x_0 \to 0} \frac{f(x_0, z; t - s)}{f(x_0, x; t)} & = & \lim_{x_0 \to 0} \frac{f_1(x_0, z; t - s)}{f_1(x_0, x; t)} \nonumber \\
	& = & \frac{\mu(t - s)}{\beta(t - s)} \frac{\beta(t)}{\mu(t)} e^{-z/\beta(t - s)} e^{x/\beta(t)}, 
\end{eqnarray*}
and the integral is 
\begin{eqnarray*}
\lefteqn{\int_0^\infty f(z, \zeta; s - u)f_1(\zeta, x, u) d\zeta} \nonumber \\
	& = & \int_0^\infty \sum_{l = 1}^\infty \frac{(z\mu(s - u)^l}{l!} e^{-z\mu(s - u)} \frac{\zeta^{l - 1}}{\beta(s - u)^l (l - 1)!} e^{-\zeta/\beta(s - u)} \nonumber \\
	&&			\qquad\qquad	\times \; \zeta\mu(u) e^{-\zeta\mu(u)} \frac{1}{\beta(u)} e^{-x/\beta(u)} d\zeta \nonumber \\
	& = & \sum_{l = 1}^\infty \frac{1}{(l - 1)!} \left(\frac{z\mu(s - u)}{\beta(s - u)} \right)^l \frac{\mu(u)}{\beta(u)} e^{-z\mu(s - u) - x/\beta(u)} 
					\frac{1}{l!} \int_0^\infty \zeta^l e^{-\zeta U(s, u)} d\zeta \nonumber \\  
	& = & \frac{\mu(u)}{\beta(u)}\frac{1}{U(s, u)^2}\frac{\mu(s - u)}{\beta(s - u)} z e^{-z\mu(s - u) - x/\beta(u)} 
	                                 \exp\left\{z\frac{\mu(s - u)}{\beta(s - u)} \frac{1}{U(s, u)} \right\} \nonumber \\
	& = & \frac{\mu(s)}{\beta(s)} z e^{-z\mu(s)} e^{- x/\beta(u)}, 
\end{eqnarray*}
using the identities Eqs.~(\ref{UIdentity1}) and (\ref{UIdentity2}) in the last line.  Then
\[
H(t, s, u) = \frac{\mu(s)\beta(t)\mu(t - s)}{\beta(s)\mu(t)\beta(t - s)} ze^{-zU(t, s)} e^{-x(\beta(u)^{-1} - \beta(t)^{-1})}. 
\]
The joint density of $T_2$ and $X(t - T_2)$ is obtained from 
\begin{eqnarray*}
\lefteqn{\mathbb{P}(X(t - s) \in (z, z + dz), T_2 \in (s - ds, s) \mid X(t) = x, T_1 = t)}  \nonumber \\
	& = & \mathbb{P}(X(t - s) \in (z, z + dz), A(s) = 1 \mid X(t) = x, T_1 = t) \nonumber \\
	&&		\quad - \;\mathbb{P}(X(t - s) \in (z, z + dz), A(s - ds) = 1 \mid X(t) = x, T_1 = t) \nonumber \\
	& = & (H(t, s, s) - H(t, s, s - ds)) dz \nonumber \\
	& = & \left. \frac{\partial H(t, s, u)}{\partial u} \right|_{u = s} ds dz \nonumber \\
	& = & \frac{1}{2} \frac{\mu(s)^2\beta(t)\mu(t - s)}{\beta(s)^2\mu(t)\beta(t - s)} zxe^{-zU(t, s)} e^{-x(\beta(s)^{-1} - \beta(t)^{-1})} dsdz,  
\end{eqnarray*} 
and Eq.~(\ref{jointMRCAgivenT1}) follows with help from Eqs.~(\ref{mubetaIdentity4}) and (\ref{UIdentity3}).  
\end{proof}
\begin{remark}
As a consistency check, integrate out $z$ from Eq.~(\ref{jointMRCAgivenT1}) and simplify with Eq.~(\ref{UIdentity2}) to give the marginal density of $T_2$, 
\[
f_{T_2 \mid T_1=t, X(t) = x} (s) = \frac{1}{2} \frac{\mu(s; |\alpha|)}{\beta(s; |\alpha|)} x e^{-x(\beta(s; |\alpha|)^{-1} - \beta(t; |\alpha|)^{-1})},
\]
which is consistent with Eq.~(\ref{densityTkGivenx}) with $k = 2$.  
\end{remark}

%
%
\subsection{$T_1 \to \infty$ limit of the population MRCA}	
\label{sec:TInfLimitMRCA}

As in Section~\ref{sec:TInfLimit} we consider the $T_1 \to \infty$ limit 
of the joint $(T_2, X_{\rm MRCA})$ distribution corresponding to a population with a single ancestral line stretching infinitely far into the past.  
\begin{theorem}
Consider a Feller diffusion $\big(X(\tau)\big)_{\tau \in [0, t]}$ with parameter $\alpha \in \mathbb{R}$ descended from a single founder at 
an infinite time in the past and 
conditioned on a current population $X(t) = x > 0$.  The limiting joint density of the time $T_2$ back to the MRCA of the current population and 
contemporaneous population size $X_{\rm MRCA}$ is 
\begin{eqnarray}	\label{jointMRCAinfT1}
f^{{\rm inf}\,T_1}_{T_2, X_{\rm MRCA}}(s, z; x) & := & \lim_{t \to \infty}f_{T_2, X_{\rm MRCA}}(s, z; t, x) \nonumber \\
	& = & \frac{\mu(s; |\alpha|)}{2\beta(s; |\alpha|)} xe^{-x/\beta(s; |\alpha|)} \times \mu(s; |\alpha|)^2 ze^{-z\mu(s; |\alpha|)} , \qquad s, z, x > 0. \nonumber \\
\end{eqnarray}
\end{theorem}
\begin{proof}
For $|\alpha| \ge 0$, 
\[
\lim_{t \to \infty} \frac{\mu(t - s; |\alpha|)}{\mu(t; |\alpha|)} = 1, \quad \lim_{t \to \infty} \frac{\beta(t; |\alpha|)}{\beta(t - s; |\alpha|)} = e^{|\alpha| s} = \mu(s; |\alpha|)\beta(s; |\alpha|), 
\]
$\lim_{t \to \infty} U(t, s; |\alpha|) = \mu(s; |\alpha|) $ and $\lim_{t \to \infty}\beta(t; |\alpha|)^{-1} = 0$.  The result then follows from Eq.~(\ref{jointMRCAgivenT1}). 
\end{proof}

In the critical case, $\mu(s; 0) = \beta(s;0)^{-1} = 2/s$ yields 
\[
f^{{\rm inf}\,T_1}_{T_2, X_{\rm MRCA}}(s, z; x) = \frac{8}{s^4} zxe^{-2(x + z)/s}, \qquad s, z > 0; \quad \alpha = 0. 
\]
As pointed out in Subsection~\ref{sec:DiffusionLimits}, for non-critical diffusions the parameter $\alpha$ can always be scaled to 1.  Equivalently, numerical 
results can conveniently be presented in terms of dimensionless times $\alpha t$ and population sizes $\alpha x$, and joint densities rendered dimensionless 
by multiplication by $\alpha^{-2}$.  

\begin{figure}
 \centering
 \begin{minipage}{.5\textwidth}
   \centering
   \includegraphics[width=0.9\linewidth]{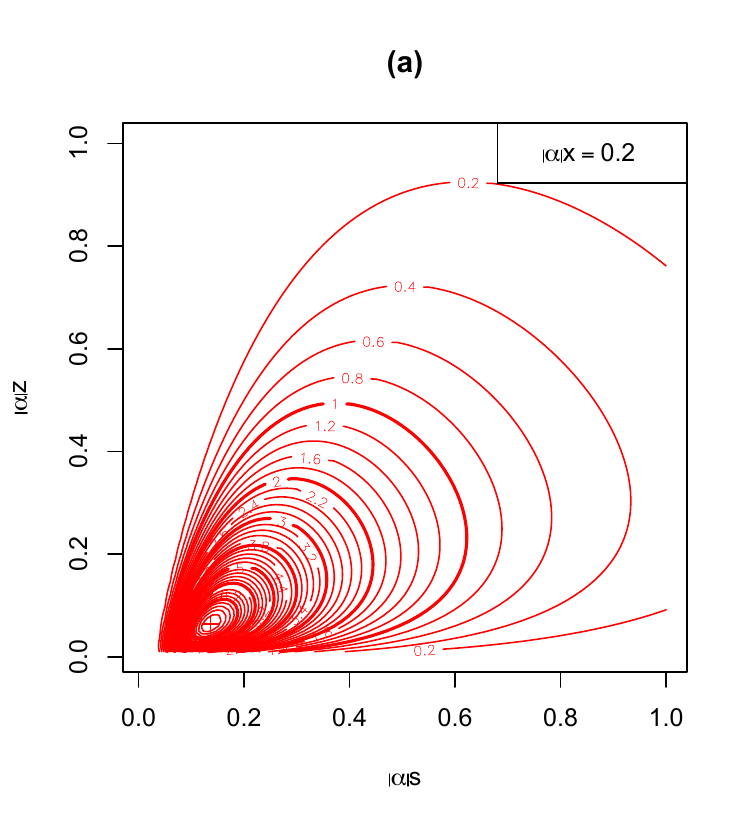}
  \end{minipage}%
 \begin{minipage}{.5\textwidth}
   \centering
   \includegraphics[width=0.9\linewidth]{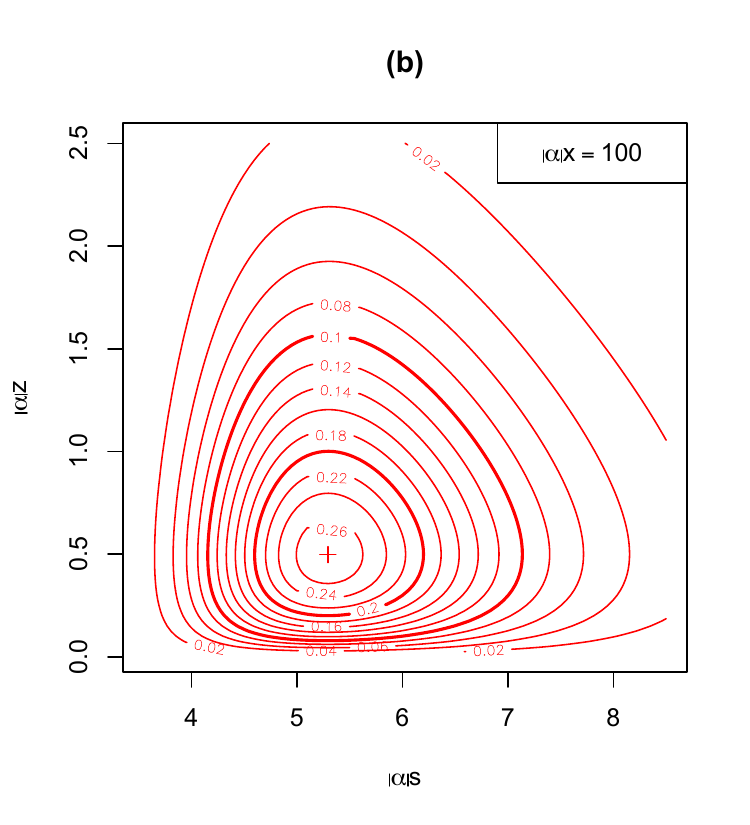}
  \end{minipage}
 \centering
 \begin{minipage}{.5\textwidth}
   \centering
   \includegraphics[width=0.9\linewidth]{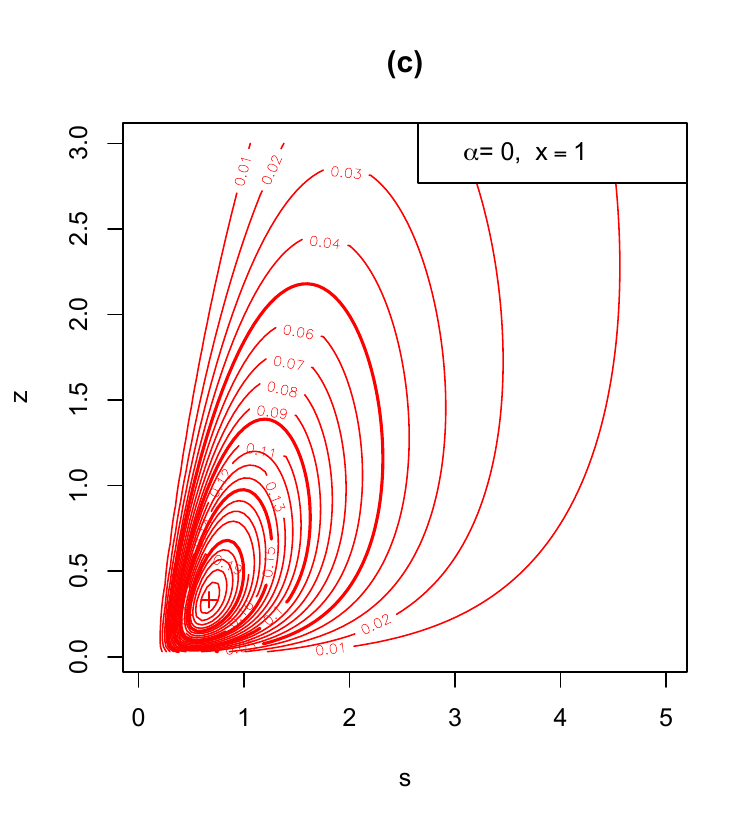}
  \end{minipage}%
 \begin{minipage}{.5\textwidth}
   \centering
   \includegraphics[width=0.9\linewidth]{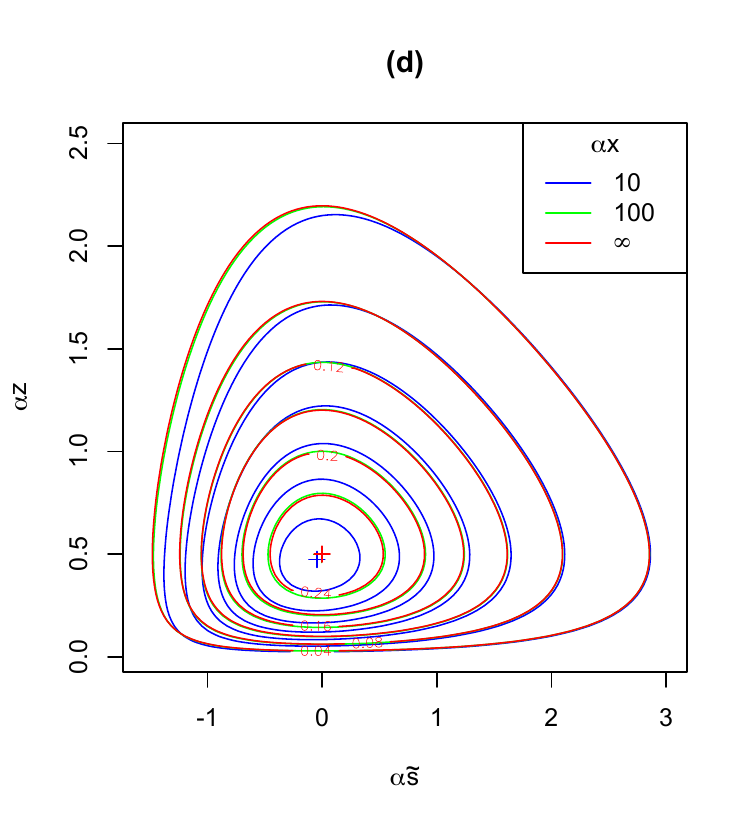}
  \end{minipage}
 \caption{
 Contour plots from the joint density Eq.~(\ref{jointMRCAinfT1}) of $(T_2, X_{\rm MRCA})$ 
 given that a population descended from a single founder at a time $T_1 \to \infty$ in the past has a current population~$x$. 
 (a, b) Contours of the dimensionless scaled density $\alpha^{-2} f^{{\rm inf}\,T_1}_{T_2, X_{\rm MRCA}}(s, z; x)$ for $|\alpha| x = 0.2$ 
 and $|\alpha| x = 100$ respectively; (c) contours for the critical case $\alpha = 0$, $x = 1$; (d) contours of the dimensionless  
 scaled density of $(\tilde{T}_2, X_{\rm MRCA})$ as defined by Eq.~(\ref{shiftedS}) for $\alpha x = 10, 100$ and $\infty$. 
 }
 \label{fig:T_2_Z_contours}
 \end{figure}

Computed contour plots of dimensionless scaled density from Eq.~(\ref{jointMRCAinfT1}) are shown in Fig.~\ref{fig:T_2_Z_contours}.  
These plots convey information about the MRCA not apparent in the simulations of Fig.~\ref{fig:InfT1Trees}.  The marginal distribution of $T_2$ gives 
an indication not only of the scale of the time since the MRCA, but also the uncertainty with which that time can be known.  Likewise the 
marginal distribution in $X_{\rm MRCA}$ gives an indication of the scale and uncertainty of the population at that time, all of whose descendent lines 
have become extinct except one individual.  In any application one is typically approximating an underlying BD or BGW process with parameters 
for the analogous Feller diffusion as described in Subsection~\ref{sec:DiffusionLimits}.  An example is given at the end of \ref{sec:MRCAfromBSandBS}
of a crude calculation of the time of mitochondrial Eve (mtE) and contemporaneous population size assuming an underlying BGW process.  
While calculations of the time since mtE based on branching processes have been made 
previously~\citep{o1995genealogy,zimmerman2001population,cyran2010alternatives}, the interesting innovation here is the measure of 
uncertainty inherent in such estimates.


\subsection{MRCA density with an improper uniform prior on $T_1$}

As in Section~\ref{sec:UnifT1} we consider the joint $(T_2, X_{\rm MRCA})$ distribution  for a Feller diffusion subject to an improper 
uniform prior on the time $T_1$ since initiation of the process with a single ancestral founder of the current population.  

\begin{theorem}
For a Feller diffusion $\big(X(\tau)\big)_{\tau \in [0, t]}$ with parameter $\alpha \in \mathbb{R}$, subject to an improper uniform prior on $T_1$ 
and conditioned on a current population $X(T_1) = x > 0$, the joint density of the time back to the MRCA of the current population and 
contemporaneous population size is 
\begin{eqnarray}	\label{jointMRCAUnifT1}
f^{{\rm unif}\,T_1}_{T_2, X_{\rm MRCA}}(s, z; x) & := &f^{{\rm unif}\,T_1}_{T_2, X_{\rm MRCA} \mid X(T_1) = x}(s, z) \nonumber \\
	& = & \frac{1}{2} \left(\frac{\mu(s; |\alpha|)}{\beta(s; |\alpha|)}\right)^2 x^2 e^{-x/\beta(s; |\alpha|)} e^{-z \mu(s; |\alpha|)}, \qquad s, z > 0.
\end{eqnarray}
\end{theorem}
\begin{proof}
From Eqs.~(\ref{uniformPosteriorT1}) and (\ref{jointMRCAgivenT1}), 
\begin{eqnarray*}
\lefteqn{f^{{\rm unif}\,T_1}_{T_2, X_{\rm MRCA}}(s, z; x) } \nonumber \\
	& = &\int_s^\infty f_{T_2, X_{\rm MRCA}}(s, z; t, x) \times f^{{\rm unif}\,T_1} _{T_1} (t; x)  \,dt \nonumber \\
	& = &\int_s^\infty \frac{1}{2} \frac{\mu(s; |\alpha|)^2\beta(t; |\alpha|)\mu(t - s; |\alpha|)}{\beta(s; |\alpha|)^2\mu(t; |\alpha|)\beta(t - s; |\alpha|)}
							zxe^{-zU(t, s; |\alpha|)} e^{-x(\beta(s; |\alpha|)^{-1} - \beta(t; |\alpha|)^{-1})}  \nonumber \\
	& &	\qquad\qquad\qquad\qquad\qquad\qquad\qquad\qquad	\times \frac{x}{2} \frac{\mu(t; |\alpha|)}{\beta(t; |\alpha|)} e^{-x/\beta(t; |\alpha|)} dt  \nonumber \\
	& = & \frac{1}{2} \left(\frac{\mu(s; |\alpha|)}{\beta(s; |\alpha|)}\right)^2 x^2 e^{-x/\beta(s; |\alpha|)} e^{-z/\beta(s; |\alpha|)} 
									\int_s^\infty \frac{z}{2} \frac{\mu(t - s; |\alpha|)}{\beta(t - s; |\alpha|)} e^{-z\mu(t - s; |\alpha|)} dt  \nonumber \\
	& = & \frac{1}{2} \left(\frac{\mu(s; |\alpha|)}{\beta(s; |\alpha|)}\right)^2 x^2 e^{-x/\beta(s; |\alpha|)} e^{-z \mu(s; |\alpha|)}, \qquad s, z > 0,
\end{eqnarray*}
using Eq.~(\ref{mubetaIdentity1}) in the last line.  
\end{proof}

Fig.~\ref{fig:T_2_Z_contours_unif} shows computed contour plots of this density. 

\begin{figure}
 \centering
 \begin{minipage}{.5\textwidth}
   \centering
   \includegraphics[width=0.9\linewidth]{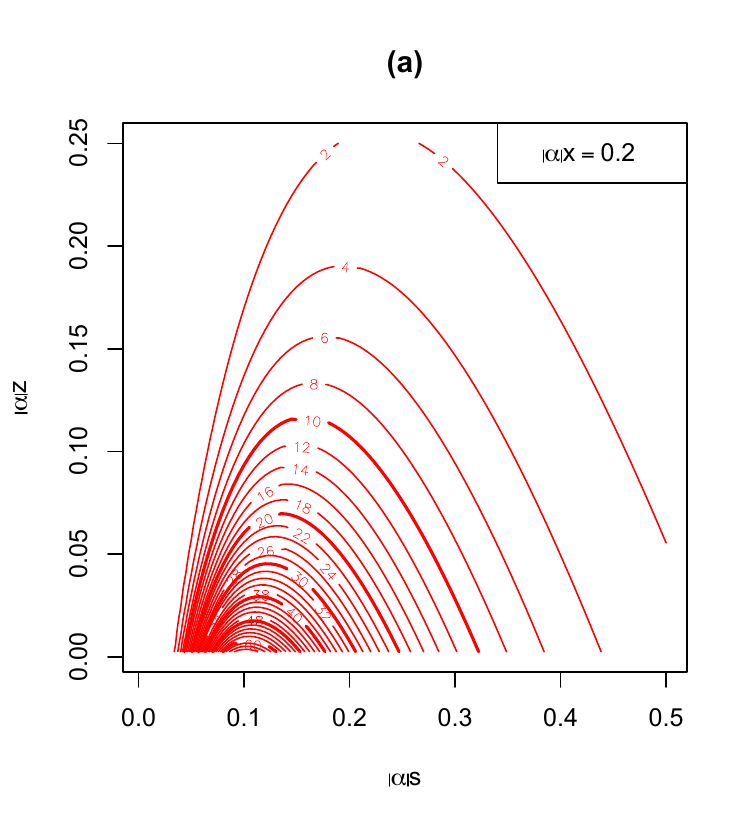}
  \end{minipage}%
 \begin{minipage}{.5\textwidth}
   \centering
   \includegraphics[width=0.9\linewidth]{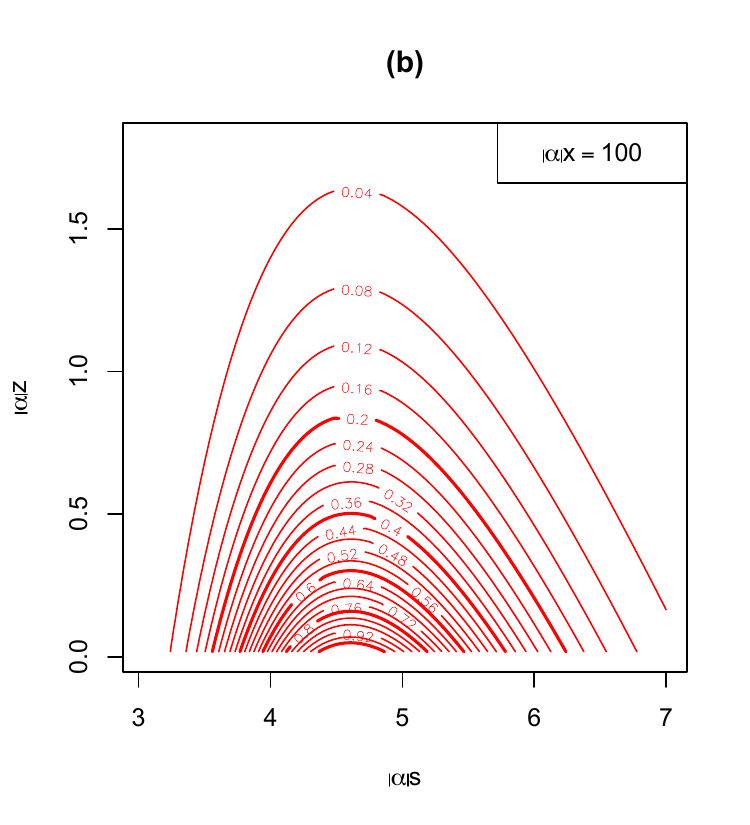}
  \end{minipage}
 \caption{Contours of the joint dimensionless scaled density from Eq.~(\ref{jointMRCAUnifT1}) of $(T_2, X_{\rm MRCA})$ subject to an improper 
uniform prior on the time $T_1$ since initiation of the process with a single ancestral founder of the current population~$X(T_1) = x$ for 
(a) $|\alpha|x = 0.2$ and (b) $ |\alpha| x = 100$.  }
 \label{fig:T_2_Z_contours_unif}
 \end{figure}

%

\section{Limiting distributions for a supercritical diffusion as $X(t) = x \to \infty$}
\label{sec:LargeXDistribs}

Heuristically, one expects any large-scale observation of a non-extinct supercritical Feller diffusion $X(t)$ satisfying 
\[
dX(t) = \sqrt{X(t)} dW(t) + \alpha X(t) dt
\]
for $\alpha > 0$ to resemble more closely a deterministic exponential growth proportional to $e^{\alpha t}$ as the second 
term on the right hand side dominates the stochastic variability from the first term for sufficiently large scaled populations $\alpha X(t) >> 1$.  
Numerical simulations of near-critical BGW processes initiated from a single founder and 
conditioned on survival in Fig.~\ref{fig:BGWsimulationLargeX} are consistent with such a scenario and suggest a long-term behaviour 
\[
\alpha X(t) \mid (A_\infty(t) = 1, X(t) > 0) \sim \tfrac{1}{2} e^{\alpha(t - C)}, 
\]
where $C$ is a random shift accounting for the time taken for the process to escape its early stochastic-dominated phase.  An {\em ad hoc} 
factor $\tfrac{1}{2}$ has been introduced here to simplify results below.\footnote{Moreover, one can show from Lemma~\ref{lemma:OConnell} 
that $\alpha C$ is asymptotically distributed as standard Gumbel.}  Here we examine the limiting behaviour of distributions of coalescent 
times conditioned on an observation $X(t) = x$ in the limit $x \to \infty$. 
\begin{figure}
 \centering
    \includegraphics[width=0.5\linewidth]{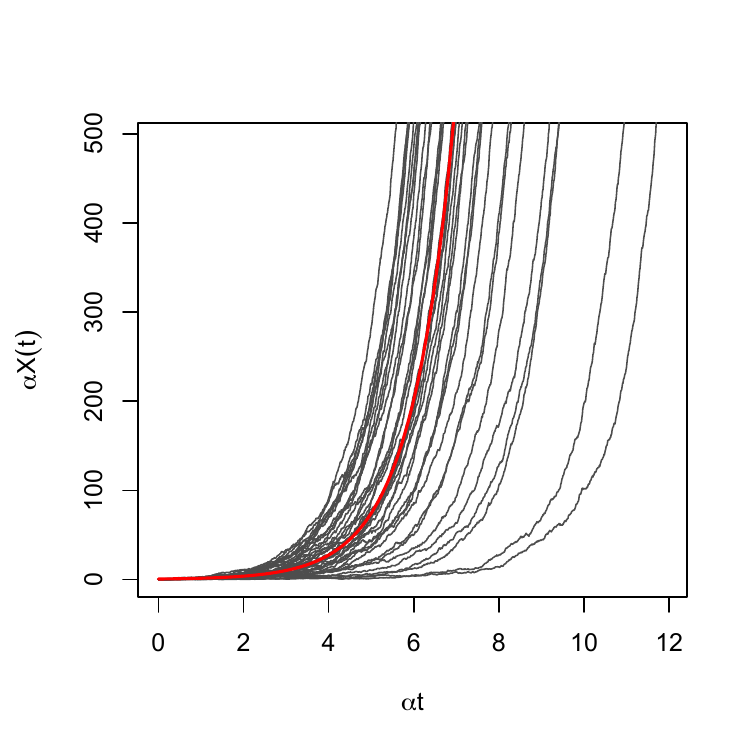}
 \caption{Independent simulations of a near-critical BGW process $\big(Y(i)\big)$ conditioned on long-term survival and given a single ancestral parent 
 $Y(0) =1$.  The number of offspring per parent is chosen as Poisson with mean and variance  $\lambda = \sigma^2 = 1.01$.  The axes 
 have been scaled to the diffusion limit continuum time $\alpha t$ and population size $\alpha X(t)$ using Eq.~(\ref{BGWscaling}). The red curve is 
 $\alpha X(t) = \tfrac{1}{2} e^{\alpha t}$.}  
 \label{fig:BGWsimulationLargeX}
 \end{figure}

{
Define a shifted time scale $\tilde{s}$ in the past and shifted coalescent times $\tilde{T}_k$ by the weak limit 
\begin{equation}	\label{shiftedS} 
\alpha \tilde{s} = \alpha s - \log 2\alpha x, \qquad \alpha \tilde{T}_k = \lim_{x \to \infty} \{(\alpha T_k \mid T_1 = t, X(t) = x) - \log 2\alpha x\},  
\end{equation}
where $s$ is the time back from the present time $t$, and coalescent times $T_k$ are defined in Subsection~\ref{sec:numberOfAncestors}.   
This shift has the effect of translating each of the trajectories in Fig.~\ref{fig:BGWsimulationLargeX} along the time axis so that it asymptotes to the 
red curve, thus capturing only the influence of the random variable $C$.  Eq.~(\ref{shiftedS}) leads to the following identities as $x \to \infty$,  
\[
\mu(s) = 2\alpha + x^{-1} e^{-\alpha\tilde{s}} + \mathcal{O}\left(x^{-2}\right), 
\]
\[
\beta(s) = xe^{\alpha\tilde{s}} - \frac{1}{2\alpha}, 
\]
and
\[
e^{-x/\beta(s)} = \exp(-e^{-\alpha\tilde{s}})\left(1 + \mathcal{O}(x^{-1})\right). 
\]
}  

Many of the theorems of Sections~\ref{sec:TInfLimit}, \ref{sec:UnifT1} and \ref{sec:MRCAgivenX} lead to limiting 
universal distributions describing properties of the coalescent behaviour of supercritical Feller diffusions given an observed limiting exponential 
growth of a large scaled population $\alpha X(t) >> 1$.  For example, the following theorem applies specifically to a Feller diffusion initialised 
at a time $T_1 = t \to \infty$ in the past.
\begin{theorem}	\label{theorem:LargeXDistribs}
Consider a supercritical Feller diffusion $\big(X(\tau)\big)_{\tau \in [0, t]}$ descended from a single founder at 
an infinite time in the past and conditioned on a current population $X(t) = x > 0$.  
Then in the limit $x \to \infty$, the marginal density of the shifted coalescent time $\tilde{T}_k$ is  
\[
f_{\tilde{T}_k}^{{\rm inf}\,T_1}(\tilde{s}) 
	= \frac{\alpha}{(k - 2)!} e^{-\alpha(k - 1)\tilde{s}} e^{-e^{-\alpha \tilde{s}}}, \qquad \tilde{s} \in \mathbb{R}, \quad k = 2, 3, \ldots, 
\]
the joint density of the shifted coalescent times $\tilde{T}_2, \ldots, \tilde{T}_k$ is  
\[
f_{\tilde{T}_2, \ldots, \tilde{T}_k}^{{\rm inf}\,T_1}(\tilde{s}_2, \ldots, \tilde{s}_k) =
	\left(\prod_{j = 2}^k \alpha e^{-\alpha \tilde{s}_j}\right) e^{-e^{-\alpha \tilde{s}_k}}, 
									\quad \infty > \tilde{s}_2 > \tilde{s}_3 > \ldots > \tilde{s}_k > -\infty, \quad k = 2, 3, \ldots,
\]
and the joint density of the shifted time $\tilde{T}_2$ back to the population MRCA and contemporaneous 
population size $X_{\rm MRCA}$ is 
\begin{equation} \label{jointT2tildeXMRCA}
f_{\tilde{T}_2,X_{\rm MRCA}}^{{\rm inf}\,T_1}(\tilde{s}, z) 
		= \alpha e^{-\alpha \tilde{s}} e^{-e^{-\alpha \tilde{s}}}\times 4\alpha^2ze^{-2\alpha z}, \qquad \tilde{s} \in \mathbb{R}, \quad z > 0.
\end{equation}
\end{theorem}
 \begin{proof}
 The required results follow immediately from Eqs.~(\ref{fTkInfT1}), (\ref{joint_fTkInfT1}), and (\ref{jointMRCAinfT1})  respectively 
by making use of the above identities for the limiting asymptotic forms of the functions $\mu(s)$ and $\beta(s)$.  
\end{proof}
\begin{remark}
These results can in principle be extended to critical and subcritical diffusions by replacing $\alpha$ by $|\alpha|$.  However, such an extension 
is of little practical value as the probability of observing a large value of $\alpha x$ in these cases is exceedingly small.
\end{remark}

Note from Eq.~(\ref{jointT2tildeXMRCA}) that $\tilde{T}_2$ and $X_{\rm MRCA}$ are independent random variables. The limiting marginal distribution
of $\alpha\tilde{T}_2$ is standard Gumbel and the limiting marginal distribution of $2\alpha X_{\rm MRCA}$ is gamma with shape parameter 2 
and scale parameter 1. Figure~\ref{fig:T_2_Z_contours}(d) shows contours for the dimensionless scaled density of $(\tilde{T}_2, X_{\rm MRCA})$ for 
$\alpha x = 10, 100$ and $\infty$.  We have observed that the contours for $\alpha x > 500$ are visibly indistinguishable from those for the limiting density 
on the scale of this plot.  In \ref{sec:MRCAfromBSandBS} we compare an estimate of the time since mtE and 
contemporaneous population size using Eq.~(\ref{jointT2tildeXMRCA}) with a calculation due to \citet{burden2016genetic} using an alternative method 
based on likelihood calculations.  

Analogous results for coalescent times subject to an improper uniform prior on $T_1$ can be obtained in a similar fashion.  
\citet[Section~4.4]{Ignatieva20} obtain related results for the time back to the origin of a near-critical BD process in the limit of a large population.

%

\section{Conclusions}
\label{sec:Conclusions}

In this paper we have re-derived results from our previous paper~\citep{BurdenGriffiths24} and derived a number of new results pertaining to the coalescent 
structure of a Feller diffusion.  Our novel approach is based on directly interpreting Feller's solution to the forward Kolmogorov 
equation as a sum over the number of ancestral founders of the process (see Eq.~(\ref{Poisson_Gamma})).  

In particular we give new and more concise proofs in Section~\ref{sec:UnifT1} of the distributions of coalescent times and numbers of ancestors derived by 
\citet[theorem~4]{BurdenGriffiths24}.  The earlier results relied on the construction of a coalescent tree as the limit of a reversed reconstructed process 
associated with a BD process, and as such relied on imposing an improper uniform prior on the time of initiation of the process with a single ancestral 
progenitor.  Our new approach is broader in that it will in principle accommodate any prior assumption on initiation of the diffusion.  
In Sections~\ref{sec:singleAncestor} and \ref{sec:TInfLimit} respectively distributions based on an assumption of an initial founder at a fixed time in the past 
and infinitely far in the past are derived.  In Section~\ref{sec:improperUnifX0} an improper uniform prior on the 
initial population size at a fixed time in the past is assumed.  Results include the distribution of the number of ancestors at a given time in the past and 
the joint distribution of coalescent times for both the entire population and a random sample.  

Also derived in Section~\ref{sec:MRCAgivenX} are contour plots of the joint distribution of the time $T_2$ since the MRCA of a population and the population 
size $X_{\rm MRCA}$ at the time of that MRCA under assumptions of an individual founder at a specified time in the past, including infinitely far in the past, 
and of a uniform prior on the time since an individual founder.  
These calculations essentially answer a question posed by \cite{BurdenSoewongsoso19}, namely, the question of what can be known about the time and 
contemporaneous population size of the MRCA of a currently observed population generated by a Feller diffusion.   In general, the question will admit of a solution 
in the form of a posterior joint probability density over $T_2$ and $X_{\rm MRCA}$ given any well defined prior distribution on the time since the founding progenitor. 

For the most part, the theorems proved throughout the paper extend results to supercritical, critical and subcritical cases.  An interesting property of all 
theorems predicated on descent from a single founding ancestor and conditioned on a current observation $X(t) = x$ proved herein is 
that they exhibit symmetry between supercritical and subcritical cases.  That is to say, results depend on the growth rate parameter $\alpha$ only via $|\alpha|$.  
\citet{Maruyama74} makes a similar observation solving the backward Kolmogorov equation for the age of a rare mutation, and \citet[Section~6.1]{Wiuf18} 
and \citet{Stadler19} make similar observations for swapping birth and death rates in constant rate BD processes, 
including extending the symmetry properties to Bernoulli sampled populations.  
A supercritical diffusion which survives extinction will eventually tend towards unbounded exponential growth.  Asymptotic distributions describing the coalescence 
properties of this scenario conditioned on an observed population $X(t) = x$ are derived in Section~\ref{sec:LargeXDistribs}.  These distributions 
are universal up to a shift proportional to $\log \alpha x$.  

In future work we hope to apply knowledge of the distributional properties of coalescent trees to studying the transient behaviour of multi-type branching 
diffusions.  

%
\appendix

%
%

\section{Identities involving the functions $\mu(t; \alpha)$ and $\beta(t; \alpha)$}
\label{sec:MuBetaIdentities}

The following identities involving the functions $\mu(t; \alpha)$ and $\beta(t; \alpha)$ are used repeatedly throughout the paper.  
\begin{equation}	\label{mubetaIdentity1}
\mu(t)\beta(t) = e^{\alpha t}, \qquad \mu(t) - \frac{1}{\beta(t)} = 2\alpha, \qquad \frac{d}{dt} \mu(t) = \frac{d}{dt} \frac{1}{\beta(t)} = -\frac{1}{2} \frac{\mu(t)}{\beta(t)},  
\end{equation}
\begin{equation*}	
\left(\frac{1}{\beta(s)} - \frac{1}{\beta(t)}\right) \left(\frac{1}{\beta(t - s)} - \frac{1}{\beta(t)}\right) = \frac{\mu(t)}{\beta(t)}.
\end{equation*}
If 
\begin{equation}	\label{UtsDefn}
U(t, s) := \beta(t - s)^{-1} + \mu(s) = U(t, t - s), 
\end{equation}
then 
\begin{equation}	\label{UIdentity1}
\frac{\mu(s)}{\beta(s)}\frac{1}{U(t, s)} = \frac{1}{\beta(s)} - \frac{1}{\beta(t)} = \mu(s) - \mu(t), 
\end{equation}
and 
\begin{equation}	\label{UIdentity2}
\frac{\mu(s)}{\beta(s)}\frac{1}{U(t, s)^2}\frac{\mu(t - s)}{\beta(t - s)} = \frac{\mu(t)}{\beta(t)}. 
\end{equation}
The following symmetries are useful for relating supercritical and subcritical diffusions: 
 \begin{equation}	\label{mubetaIdentity3}
\mu(t; \alpha) \beta(t; -\alpha) = 1;
\end{equation}
\begin{equation}	\label{mubetaIdentity4}
\frac{\mu(t; \alpha)}{\beta(t; \alpha)} = \frac{\mu(t; |\alpha|)}{\beta(t; |\alpha|)}, \qquad 
	\frac{1}{\beta(s; \alpha)} - \frac{1}{\beta(t; \alpha)} = \frac{1}{\beta(s; |\alpha|)} - \frac{1}{\beta(t; |\alpha|)};  
\end{equation}
\begin{equation}	\label{UIdentity3}
U(t, s; \alpha) = U(t, s; |\alpha|).
\end{equation}

%
%

\section{Alternate forms of the transition density}\label{sec:Alternate_f}

An alternate form of Eq.~(\ref{Poisson_Gamma}) in terms of the modified Bessel function 
\[
I_\nu(z) = \left(\tfrac{1}{2}z\right)^\nu \sum_{k = 0}^\infty \frac{\left(\tfrac{1}{4}z^2\right)^k}{k! \Gamma(\nu + k + 1)} 
\]
is~\citetext{\citealp[Eq.~5.2]{feller1951diffusion}; \citealp[Eq.~6.2]{feller1951two}}
\begin{equation}	\label{modBesselDensity}
f(x_0, x; t) = e^{-x_0 \mu(t)} \delta(x) + e^{-x_0 \mu(t) - x/\beta(t)} \sqrt{\frac{x_0\mu(t)}{x\beta(t)}} I_1 \left( 2\sqrt{\frac{xx_0\mu(t)}{\beta(t)}}\right). 
\end{equation}

The transition density Eq.~(\ref{Poisson_Gamma}), or equivalently Eq.~(\ref{modBesselDensity}), can also be written in terms of 
the 1-parameter family of density functions 
\begin{eqnarray*}	
f_{\rm CPE}(z; \xi) & = & \delta(z) e^{-\xi} + \sum_{l = 1}^\infty \frac{e^{-\xi} \xi^l}{l!} \frac{1}{\Gamma(l)} \xi^l z^{l -1} e^{-\xi z}		\nonumber \\
	& = & \delta(z) e^{-\xi} + \xi e^{-\xi(1 + z)} \frac{I_1\left(2 \xi \sqrt z \right)}{\sqrt z}, \quad z \ge 0, \xi > 0.  
\end{eqnarray*}
This is the density of a compound Poisson exponential random variable 
\[
Z_\xi = \sum_{i = 1}^L U_i, 
\]
where $L$ is a Poisson random variable with mean $\xi$ and the $U_i$ are i.i.d.\ exponential random variables which are independent of $L$ and 
with common mean $\xi^{-1}$.  One easily checks that 
\begin{equation*}	\label{fInTermsOfCPE}
f(x_0, x; t) = \frac{1}{x_0 \mu(t) \beta(t)} f_{\rm CPE}\left(\frac{x}{x_0 \mu(t) \beta(t)}, x_0 \mu(t) \right).
\end{equation*}
Equivalently, 
\[
X(t) = x_0 \mu(t)\beta(t) Z_{x_0\mu(t)} = x_0 e^{\alpha t} Z_{x_0\mu(t)}.  
\]
In \citet[Eq.~(32) with a typographical error]{burden2018mutation} and \citet[Eq.~(24)]{BurdenSoewongsoso19}, the density $f_{\rm CPE}$ is referred to as $f_{\rm Feller}$.  

%
%

\section{Connection with \citet{o1995genealogy} and \citet{Harris20}}
\label{sec:connectionOC_HJR}

At the end of Section~\ref{sec:SamplingDistributions} it is noted that the density of the MRCA of a sample of size $n = 2$ without 
conditioning of $X(t) = x$, namely Eq.~(\ref{MRCAnEquals2}) has been found previously by \citet{o1995genealogy} and \cite{Harris20} 
by considering the near-critical limit of a BD process.  That their limiting process is identical to that defined by Eq.~(\ref{BDlimit}) in 
Section~\ref{sec:DiffusionLimits} is not immediately obvious.  Here we fill in the details.  It is easiest to begin with \cite{Harris20}.  

In the following we will append a subscript HJR to parameters whose name clashes with notation used within this appendix.   
\cite{Harris20} define the near critical limit of a BP defined on an interval $[0, T]$ with mean number $L$ of offspring per parent at 
each branching event 
\[
\mathbb{E}_T[L] = 1 + \frac{\mu}{T} + o\left(\frac{\mu}{T}\right), 
\]
by taking a limit $T \to \infty$ while holding $\mu$ fixed.  By contrast, in the current paper we define the BD on a fixed interval $[0, t]$, 
define birth and death rates by Eq.~(\ref{BDlimit}) and take the limit $\epsilon \to 0$ with $\alpha$ fixed.  The mean number of offspring in our treatment is 
\[
\mathbb{E}_\epsilon[L] = 2 \times \mathbb{P}(L = 2) + 0\times \mathbb{P}(L = 0) = \frac{2\hat{\lambda}(\epsilon)}{\hat{\lambda}(\epsilon) + \hat{\mu}(\epsilon)} 
							= 1 + \alpha\epsilon + \mathcal{O}(\epsilon^2), 
\]
Matching the two means gives 
\begin{equation} \label{muOverTScaling}
\frac{\mu}{T} = \alpha\epsilon + \mathcal{O}(\epsilon^2).  
\end{equation}

\cite{Harris20} also define a parameter $r = \alpha_{\rm HJR} + \beta_{\rm HJR}$, where $\alpha_{\rm HJR}$ and $\beta_{\rm HJR}$ are branching rates on 
the timescale of the interval $[0, T]$.  To preserve the near-critical limit  
$\mathbb{E}_T[L] = 2\beta_{\rm HJR}/r \to 1$, it must be that $\beta_{\rm HJR}/r, \alpha_{\rm HJR}/r \to \tfrac{1}{2}$ as $T \to \infty$.    
It  follows the $\lim_{T \to \infty}{\rm Var}_T(L) = 4 \alpha_{\rm HJR} \beta_{\rm HJR}/r^2 = 1$.  
It is straightforward to check from Eq.~(\ref{BDlimit}) that our diffusion limit also yields $\lim_{\epsilon \to 0}{\rm Var}_\epsilon(L) = 1$.  

In order to preserve the total probability of a branching event occuring over any interval in $[0, T]$ or the corresponding interval in $[0, t]$ we have 
\begin{equation}	\label{rTScaling}
rT = (\hat{\lambda}(\mu) + \hat{\mu}(\mu)) t = \frac{t}{\epsilon} + \mathcal{O}(\epsilon).  
\end{equation}
Since both limiting processes have the same offspring-number mean and variance, and making the usual assumption 
that all higher moments tend to zero in the diffusion limit, both ways of defining a near-critical limit lead to a Feller diffusion with parameter $\alpha$ 
providing the scaling laws Eqs.~(\ref{muOverTScaling}) and (\ref{rTScaling}) are respected.  

Coalescent times $\mathcal{S}^n_k(T),$ in \citet[p1371]{Harris20} for a sample of size $n$ are measured forward in time from initiation of the process.  
They are related to our coalescent times $T_k^{(n)}$ by (see Fig.~\ref{fig:CoalescentTree}) 
\[
\frac{\mathcal{S}^n_k(T)}{T} = t - T_{k + 1}^{(n)}, \qquad  k = 1, \ldots, n -1.  
\]
In particular, \cite{Harris20} quote O'Connell's distribution of split times $\mathcal{S} = \mathcal{S}^2_1$ for a sample of size $n = 2$ as 
\begin{eqnarray*}
\lefteqn{\lim_{T \to \infty} \mathbb{P}_T\left(\frac{\mathcal{S}(T)}{T} \ge s_{\rm HJR} \mid N_T \ge 2 \right)} \nonumber \\
	& = & 2\left(\frac{e^{r\mu(1 - s_{\rm HJR})} -1}{e^{r\mu(1 - s_{\rm HJR})} - e^{r\mu}}\right)
 +  2\frac{(e^{r\mu} -1)(e^{r\mu(1 - s_{\rm HJR})} - 1)}{(e^{r\mu(1 - s_{\rm HJR})} - e^{r\mu})^2} \log\left(\frac{e^{r\mu} - 1}{e^{r\mu(1 - s_{\rm HJR})} - 1}\right). 
\end{eqnarray*}
This is seen to be equivalent to Eq.~(\ref{MRCAnEquals2}) using the correspondences in Table~\ref{tab:HOparameters}, which follow from 
Eqs.~(\ref{muOverTScaling})) and (\ref{rTScaling}).  The third column in Table~\ref{tab:HOparameters} is obtained by matching the notation of 
\citet{Harris20} with that of \citet[Section~2]{o1995genealogy}.  The equivalence of Eq.~(\ref{MRCAnEquals2}) and \citet[Theorem~2.3 with $x=1$]{o1995genealogy} 
follows.  

\begin{table}[htbp]
   \centering
   \caption{Matching the notation of the current paper with that of \cite{Harris20} and \citet{o1995genealogy}.  Note that our parameter $s \in [0, t]$ is measured back from the 
   current time as in Fig.~\ref{fig:CoalescentTree}, whereas the parameter $s_{\rm HJR} \in [0, 1]$ of \cite{Harris20} or $r \in [0, 1]$ of \citet{o1995genealogy} is measured forward from 
   initiation of the process.}
   \begin{center}
   \begin{tabular}{lll} 
      \hline
	Current paper		&  \cite{Harris20}	& \citet{o1995genealogy}	\\
      \hline
	$\alpha t$			& $r\mu$			& $\alpha$ 			\\
	$\alpha s$			& $r\mu(1 - s_{\rm HJR})$		& $\alpha(1 - r)$ 		\\
	$\dfrac{\beta(s)}{\beta(t)} = \dfrac{e^{\alpha s} - 1}{e^{\alpha t} - 1}$ & $\dfrac{e^{r\mu(1 - s_{\rm HJR})} - 1}{e^{r\mu} - 1}$
							& $q_r = \dfrac{e^{-\alpha r} - e^{-\alpha}}{1 - e^{-\alpha}}$	\\
      \hline
   \end{tabular}
    \label{tab:HOparameters}
    \end{center}
\end{table}

%
%

\section{Connection with \citet{burden2016genetic} and \citet{BurdenSoewongsoso19}}
\label{sec:connectionBSandBS}

\citet{burden2016genetic} and \citet{BurdenSoewongsoso19} address problems closely related to certain topics covered in the current paper.  
Firstly, they derive probabilities that Feller diffusions conditioned on both their initial and final populations have only a single ancestor 
in the initial population.  In \ref{sec:ancestorsGiven_x0andx} below we give proofs of their theorems which are more concise, and which extend results 
to any number of initial ancestors.  Secondly, they apply their results to the problem of estimating the time of the MRCA and contemporaneous 
population size given a current population.  In \ref{sec:MRCAfromBSandBS} we compare their methods with those of Sections~\ref{sec:MRCAgivenX} 
and \ref{sec:LargeXDistribs} of the current paper.  To convert from the notation in the two earlier papers to that of the current paper, for the Feller diffusion process 
stated in the notation of the current paper as $\big(X(\tau)\big)_{\tau \in [0, t]}$ with $X(0) = x_0$, $X(t) = x$ for instance, make the substitutions 
\begin{equation*}
K(s) \leadsto 2\alpha X(t), \quad \kappa \leadsto 2\alpha x, \quad \kappa_0 \leadsto 2\alpha x_0, \quad s \leadsto \alpha t.  
\end{equation*}

%

\subsection{Distributions of the number of ancestors given $X(0)$ and $X(t)$}
\label{sec:ancestorsGiven_x0andx}

\begin{theorem}
For a Feller diffusion $\big(X(\tau)\big)_{\tau \in [0, t]}$ conditioned on $X(0) = x_0 > 0$ and $X(t) = x > 0$, the probability that the final population has 
$l$ initial ancestors is 
\begin{equation}	\label{ancestorsGivenX0X}
\mathbb{P}(A_\infty(t)  = l \mid X(0) = x_0, X(t) = x) = \frac{1}{l!(l - 1)!} \frac{w^{2l - 1}}{I_1(2w)}, \qquad l = 1, 2, \ldots, 
\end{equation}
and the probability that an i.i.d.\ sample of size $n$ of the final population has $k$ initial ancestors is 
\begin{equation}	\label{sampleGivenX0X}
\mathbb{P}(A_n(t)  = k \mid X(0) = x_0, X(t) = x) = \frac{n!}{k!} \binom{n - 1}{k - 1}\frac{1}{w^{n - k}} \frac{I_{k + n - 1}(2w)}{I_1(2w)}, \qquad k = 1, 2, \ldots n, 
\end{equation}
where 
\begin{equation*}	
w := \sqrt{\frac{xx_0\mu(t)}{\beta(t)}} = \frac{\alpha \sqrt{xx_0}}{\sinh \frac{1}{2}\alpha t}. 
\end{equation*}
\end{theorem}
\begin{proof}
Conditional on an initial population $x_0 > 0$ and a final population $x > 0$, the distribution of the number of initial ancestors 
at time 0 of the final population is, from Eqs.~(\ref{flDef}) and (\ref{modBesselDensity}) 
\begin{eqnarray*}
\mathbb{P}(A_\infty(t)  = l \mid X(0) = x_0, X(t) = x)
	& = & \frac{\mathbb{P}(A_\infty(t)  = l, X(t) \in(x, x + dx) \mid X(0) = x_0)}{\mathbb{P}(X(t) \in(x, x + dx) \mid X(0) = x_0)}  \nonumber \\
 	& = & \frac{f_l(x_0, x; t)}{f(x_0, x; t)} \nonumber \\
	& = &  \frac{\frac{(x_0 \mu(t))^l}{l!} \frac{x^{l - 1}}{\beta(t)^l (l - 1)!}}
				{\sqrt{\frac{x_0\mu(t)}{x\beta(t)}} I_1 \left( 2\sqrt{\frac{xx_0\mu(t)}{\beta(t)}}\right)} \nonumber \\
	& = & \frac{1}{l!(l - 1)!} \frac{w^{2l - 1}}{I_1(2w)}.  
\end{eqnarray*} 

From Lemma~\ref{lemma:probAnGivenAinf} and for $k \in \{1, \ldots, n\}$ we have the joint probability 
\begin{eqnarray*}
\lefteqn{\mathbb{P}(A_n(t) = k, X(t) \in(x, x + dx) \mid X(0) = x_0) } \\
	& = & \sum_{l = k}^\infty \mathbb{P}(A_n(t) = k \mid A_\infty(t) = l)   \\
	& & \qquad\qquad\qquad\qquad \times \, \mathbb{P}(A_\infty(t) = l, X(t) \in(x, x + dx) \mid X(0) = x_0) \\
	& = & \sum_{l = k}^\infty \binom{l}{k} \frac{n!}{l_{(n)}} \binom{n - 1}{k - 1} \times 
							f_l(x_0, x; t) dx. 
\end{eqnarray*}
Then in a sample of size $n$ at time $t$, the probability that the sample has $1 \le k \le n$ ancestors at time 0 given an initial 
population $x_0 > 0$ and final population $x > 0$ is 
\begin{eqnarray*}	
\lefteqn{\mathbb{P}(A_n(t) = k \mid X(0) = x_0, X(0) = x) } \nonumber \\
	& = & \frac{\mathbb{P}(A_n(t, t) = k, X(t) \in(x, x + dx) \mid X(0) = x_0)}
			{\mathbb{P}(X(t) \in(x, x + dx) \mid X(0) = x_0) } \nonumber \\
	& = & \frac{\sum_{l = k}^\infty \binom{l}{k} \frac{n!}{l_{(n)}} \binom{n - 1}{k - 1} f_l(x_0, x; t)} {f(x_0, x; t)} \nonumber \\
	& = & \frac{\sum_{l = k}^\infty \binom{l}{k} \frac{n!}{l_{(n)}} \binom{n - 1}{k - 1}\frac{(x_0 \mu(t))^l}{l!} \frac{x^{l - 1}}{\beta(t)^l (l - 1)!}}
							{\sqrt{\frac{x_0\mu(t)}{x\beta(t)}} I_1 \left( 2\sqrt{\frac{xx_0\mu(t)}{\beta(t)}}\right)} \nonumber \\ 
	& = & \frac{n!}{k!} \frac{1}{I_1(2w)} \binom{n - 1}{k - 1} \sum_{l = k}^\infty \frac{w^{2l - 1}}{(l - k)! (l + n - 1)!} \nonumber \\
	& = & \frac{n!}{k!} \frac{1}{I_1(2w)} \binom{n - 1}{k - 1}w^{2k - 1} \sum_{j = 0}^\infty \frac{w^{2j}}{j!(k + n - 1 + j)!} \nonumber \\
	& = & \frac{n!}{k!} \binom{n - 1}{k - 1}\frac{1}{w^{n - k}} \frac{I_{k + n - 1}(2w)}{I_1(2w)}.   
\end{eqnarray*}
\end{proof}

Setting $l = 1$ in Eq.~(\ref{ancestorsGivenX0X}), the probability that the final population is descended from a single ancestor is 
\begin{equation}	\label{1AncestorGivenX0X}
\mathbb{P}(A_\infty(t)  = 1 \mid X(0) = x_0, X(t) = x) = \frac{w}{I_1(2w)},  
\end{equation}
agreeing with \citet[Eq.~(44)]{burden2016genetic} or \citet[Corollary~1]{BurdenSoewongsoso19}.  
Setting $k = 1$ in Eq.~(\ref{sampleGivenX0X}), the probability that a sample of size $n$ from the final population is descended from a single ancestor is 
\[
\mathbb{P}(A_n(t)  = 1 \mid X(0) = x_0, X(t) = x) = \frac{n!I_n(2w)}{w^{n - 1}I_1(2w)}.  
\]
This is \citet[Theorem~1]{BurdenSoewongsoso19}.  Results for $l > 1$ and $k > 1$ respectively are new results.  
Note also that, from the ascending series $I_m(2w) = (w^m/m!)(1 + \mathcal{O}(\tfrac{1}{m}))$ as 
$m \to \infty$ \citep[p375]{Abramowitz:1965sf}, one can recover Eq.~(\ref{ancestorsGivenX0X}) from Eq.~(\ref{sampleGivenX0X}).  

%

\subsection{MRCA given the current population}
\label{sec:MRCAfromBSandBS}

Instead of setting a prior distribution on the initial population $X(0)$ and calculating a posterior joint density for the initial population and time since the MRCA, 
\cite{BurdenSoewongsoso19} treat the initial population size $z$ and time $s$ since initiation of the process at a time coincident with the MRCA of the current 
population as parameters to be estimated.  Consider the process $\big(X(\tau)\big)_{\tau \in [0, s]}$.  Given an observed current population $X(s) = x$, 
they set out to determine a likelihood surface $L(s, z \mid x)$ with the property that any given subset $\Omega \subseteq \{(s, z) : s, z > 0\}$ has a probability  
\[
\mathbb{P}((s^{\rm true}, z^{\rm true}) \in \Omega) = \int_\Omega L(s, z \mid x) ds dz , 
\]
of containing the true initial conditions $(s^{\rm true}, z^{\rm true})$.

In order to model the surface, two functions of the $(s, z)$-plane are defined.  Firstly, the probability that the time since the MRCA of the process 
occurs subsequent to initiation is, from Eq.~(\ref{1AncestorGivenX0X}),   
\[
u(s, z \mid x) := \mathbb{P}(A_\infty(s) = 1 \mid X(0) = z, X(s) = x) = \frac{w}{I_1(2w)}, 
\]
where $w = \big(xz\mu(s)/\beta(s)\big)^{1/2}$.  Secondly, the probability that the current population will not exceed the observed value $x$ is 
\[
v(s, z \mid x) := \mathbb{P}(X(s) \le x \mid X(0) = z) = \int_0^x f(z, \xi; s) d\xi, 
\]
where $f(z, \xi; s)$ is the density defined in Eqs.~(\ref{Poisson_Gamma}) or (\ref{modBesselDensity}).  
By considering a mapping of the $(s, z)$-plane to the $(u, v)$-plane, \citet[Figure~2 and Proposition~1]{BurdenSoewongsoso19} argue that for 
fixed values of $z$ and $x$ the subset  
\[
\mathcal{I}_s(u_1, u_2) := \{(s, z): u_1 < u(s, z \mid x) < u_2\}, \qquad 0 \le u_1 < u_2 \le 1, 
\]
of the $(s, z)$-plane has a probability $u_2 - u_1$ of containing the time $s^{\rm true}$ since the MRCA, and that for fixed values of $s$ and $x$ the subset 
\[
\mathcal{I}_z(v_1, v_2) := \{(s, z): v_1 < v(s, z \mid x) < v_2\},  \qquad 0 \le v_1 < v_2 \le 1, 
\]
has a probability $v_2 - v_1$ of containing initial population $z^{\rm true}$.  Thus the mapping induces likelihood surface $L_{UV}(u, v \mid x)$ 
for $(u, v) \in [0, 1] \times [0, 1]$ via  
\[
L(s, z \mid x) = \left| \frac{\partial u}{\partial s}\frac{\partial v}{\partial z} - \frac{\partial u}{\partial z}\frac{\partial v}{\partial s} \right| L_{UV}(u, v \mid x), 
\]
satisfying $\int_0^1\int_0^{u_2} L_{UV}(u, v \mid x) dudv = u_2$ and $\int_0^{v_2}\int_0^1 L_{UV}(u, v \mid x) dudv = v_2$. 
The density $L_{UV}(u, v \mid x)$ is that of a 2-dimensional copula, that is, a density on the unit square $[0, 1] \times [0, 1]$ whose marginals are each 
uniform distributions~\citep[Section~2]{Nelsen:2006aa}.   

The above procedure does not uniquely define the desired likelihood surface $L(s, z \mid x)$.  To progress beyond this point 
\cite{BurdenSoewongsoso19} employ the Ansatz $L_{UV}(u, v \mid x) \equiv 1$, which they state as a conjecture that the random intervals 
$\mathcal{I}_s(u_1, u_2)$ and $\mathcal{I}_z(v_1, v_2)$ are independent.  Here we compare their numerical results with the shifted posterior joint density  
of $(\tilde{T}_2, X_{\rm MRCA})$ obtained in Section~\ref{sec:LargeXDistribs} under the prior assumption $T_1 \to \infty$ of a single founder infinitely far 
in the past and in the limit $x \to \infty$.  In order to make the comparison, consistent with Eq.~(\ref{shiftedS}) define 
\[	
\begin{split}
\tilde{u}(\tilde{s}, z) &= \lim_{x \to \infty} u(\tilde{s} + \alpha^{-1}\log 2\alpha x, z \mid x); \\
\tilde{v}(\tilde{s}, z) &= \lim_{x \to \infty} v(\tilde{s} + \alpha^{-1}\log 2\alpha x, z \mid x); \\
\tilde{L}(\tilde{s}, z) &= \lim_{x \to \infty} L(\tilde{s} + \alpha^{-1}\log 2\alpha x, z \mid x).  
\end{split}
\]
Contours of $\tilde{u}$, $\tilde{v}$ and $\tilde{L}$ scaled to dimensionless units are plotted in Fig.~\ref{fig:contourJacobianLikelihood}, together with 
the posterior density $f_{\tilde{T}_2,X_{\rm MRCA}}^{{\rm inf}\,T_1}(\tilde{s}, z)$ from Eq.~(\ref{jointT2tildeXMRCA}) similarly scaled.  We see that 
the heuristic procedure of \cite{BurdenSoewongsoso19} predicts a slightly higher estimate of the time since the the MRCA and a similar estimate of the 
contemporaneous population size compared with the posterior distribution assuming the prior assumption $T_1 \to \infty$.  

\begin{figure}
 \centering
    \includegraphics[width=0.5\linewidth]{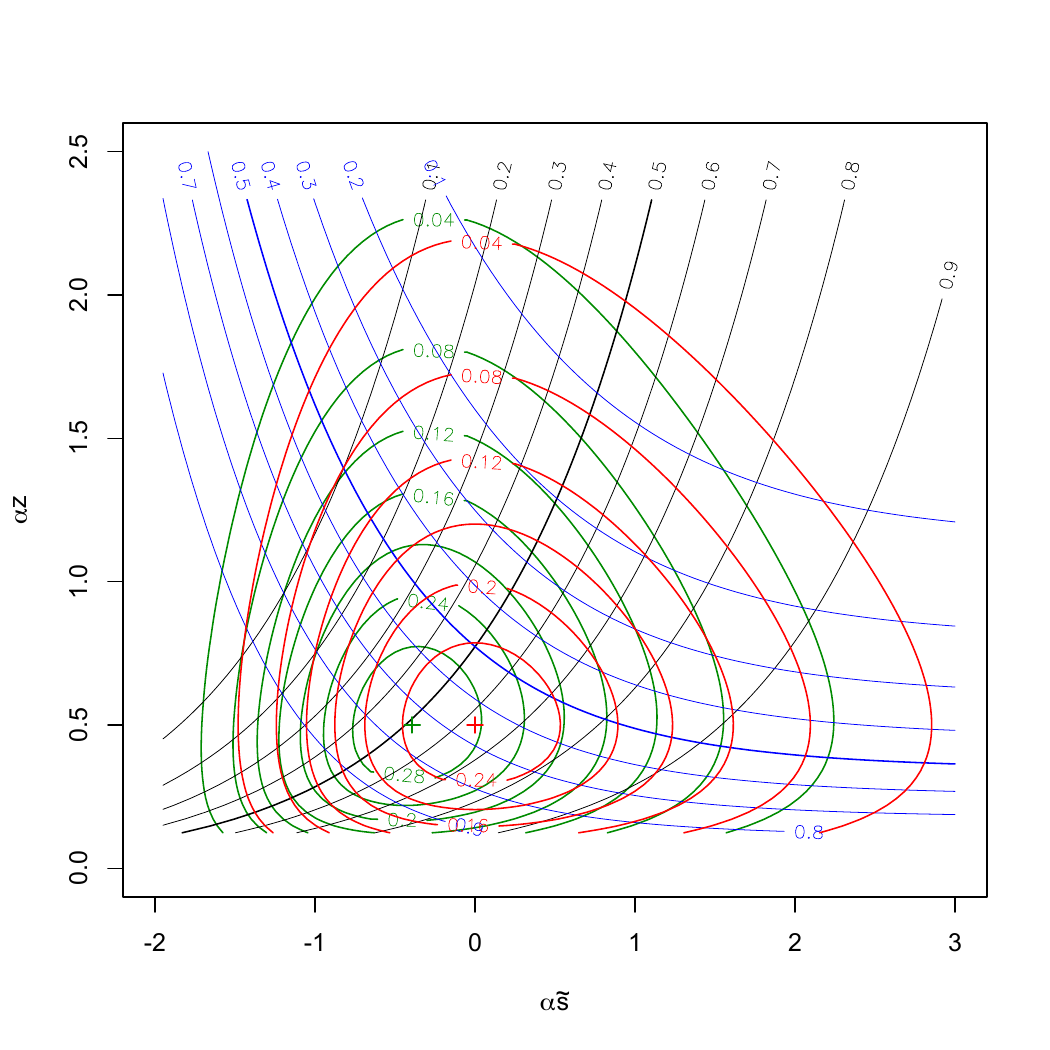}
 \caption{Comparison of the shifted likelihood surface $\alpha^{-2}\tilde{L}(\tilde{s}, z)$ (green contours), with the posterior density 
 $\alpha^{-2} f_{\tilde{T}_2,X_{\rm MRCA}}^{{\rm inf}\,T_1}(\tilde{s}, z) $ (red contours).  Also shown are the functions $\tilde{u}(\tilde{s}, z)$ (black contours) and 
 $\tilde{u}(\tilde{s}, z)$ (blue contours).}  
 \label{fig:contourJacobianLikelihood}
 \end{figure}

For completeness, we use results of the current paper to repeat an estimate of the time since mtE by \citet[Section~6]{burden2016genetic}.  The original calculation 
used a `median estimate' consisting of the intersection of the contours $\tilde{u}(\tilde{s}, z) = \tilde{v}(\tilde{s}, z) = \tfrac{1}{2}$.  Instead we use the 
shifted joint posterior density Eq.~(\ref{jointT2tildeXMRCA}), but the same assumed parameters.  With reference to the continuum scaling of the BGW process, Eq.~(\ref{BGWscaling}), 
parameters for an assumed slow exponential human population growth over the upper paleolithic era up to $10\,000$ BCE are: female population 
at the end of the period, $Y = 3 \times 10^6$, per-generation growth rate $\log\lambda = 0.0015$, and variance of the number of female offspring 
per generation $\sigma^2 = 2$, leading to a final scaled population $\alpha x = 2250$.   From Eq.~(\ref{jointT2tildeXMRCA}), 
$\mathbb{E}[\alpha \tilde{T}_2] = \gamma$, ${\rm Var}(\alpha \tilde{T}_2) = \pi^2/6$, $\mathbb{E}[\alpha X_{\rm MRCA}] = 1$, 
and ${\rm Var}(\alpha X_{\rm MRCA}) = \tfrac{1}{2}$, where $\gamma$ is the Euler-Mascheroni constant.  From Eqs.~(\ref{BGWscaling}) 
and (\ref{shiftedS}), the number $N_{\rm mtE}$ of generations since mtE and contemporaneous population size $Y_{\rm mtE}$ 
are 
\[
N_{\rm mtE} = \frac{\alpha \tilde{T}_2 + \log 2\alpha x}{\log\lambda}, \qquad Y_{\rm mtE} = \frac{\sigma^2}{\log\lambda} \alpha X_{\rm MRCA}. 
\]
Thus the mean and variance of the number of generations back from the end of the upper paleolithic to the time of mtE is 
\[
\mathbb{E}[N_{\rm mtE}] = \frac{\gamma + \log 2\alpha x}{\log\lambda} \approx 5\,990, \qquad 
	{\rm Var}(N_{\rm mtE}) = \frac{\pi^2}{6 (\log\lambda)^2} \approx 731\,000. 
\]
Assuming a generation time of 20~years, this equates to a time of approximately $120\,000$ years with a standard deviation of $17\,000$ years.  
The mean and variance of the female population at the time of mtE are 
\[
\mathbb{E}[Y_{\rm mtE}] = \frac{\sigma^2}{\log\lambda} \approx 1330, \qquad 
		{\rm Var}(Y_{\rm mtE}) = \frac{1}{2} \left(\frac{\sigma^2}{\log\lambda}\right)^2 \approx 888\,000, 
\]
or a standard deviation of approximately 942.

\bibliographystyle{elsarticle-harv}\biboptions{authoryear}

\end{document}